\documentclass[11pt]{article}
\usepackage{hyperref}
\usepackage{amsmath, graphicx, latexsym, amssymb, amsthm, amsfonts}
\usepackage[mathscr]{eucal}
\usepackage{graphics}
\usepackage{color}
\usepackage{epsfig}
\usepackage{enumerate}

\newtheorem{lemma}{Lemma}[section]
\newtheorem{theorem}[lemma]{Theorem}
\newtheorem{corollary}[lemma]{Corollary}

\newtheorem{proposition}[lemma]{Proposition}

\newtheorem{remark}[lemma]{Remark}
\newtheorem{construction}[lemma]{Construction}

\newcommand{\leb}{\mbox{leb}}

\newcommand{\half}{{\frac{1}{2}}}
\newcommand{\quarter}{{\frac{1}{4}}}

\newcommand{\Nmb}{{\mathbb{N}}}

\newcommand{\Rmb}{{\mathbb{R}}}

\newcommand{\Bmc}{{\mathcal{B}}}
\newcommand{\Cmc}{{\mathcal{C}}}
\newcommand{\Dmc}{{\mathcal{D}}}
\newcommand{\Emc}{{\mathcal{E}}}
\newcommand{\Fmc}{{\mathcal{F}}}

\newcommand{\Jmc}{{\mathcal{J}}}

\newcommand{\Mmc}{{\mathcal{M}}}

\newcommand{\Umc}{{\mathcal{U}}}

\newcommand{\cls}{{\mathcal{S}}}

\newcommand{\one}{{\boldsymbol{1}}}

\newcommand{\fbar}{{\bar{f}}}

\newcommand{\Xbar}{{\bar{X}}}
\newcommand{\xibar}{{\bar{\xi}}}

\newcommand{\Ybar}{{\bar{Y}}}

\newcommand{\Zbar}{{\bar{Z}}}
\newcommand{\zetabar}{{\bar{\zeta}}}

\newcommand{\gtil}{{\tilde{g}}}

\newcommand{\Ntil}{{\tilde{N}}}

\newcommand{\xtil}{{\tilde{x}}}
\newcommand{\Xtil}{{\tilde{X}}}
\newcommand{\xitil}{{\tilde{\xi}}}

\newcommand{\ytil}{{\tilde{y}}}

\newcommand{\zetatil}{{\tilde{\zeta}}}

\newcommand{\la}{\lambda}

\newcommand{\R}{{\mathbb R}}

\newcommand{\calC}{{\cal C}}

\newcommand{\calU}{{\cal U}}

\newcommand{\w}{\wedge}

\newcommand{\be}{\begin{equation}}
\newcommand{\ee}{\end{equation}}

\numberwithin{equation}{section}
\begin{document}

\title{Large Deviations for the Single Server Queue and the Reneging Paradox}
\author{Rami Atar\thanks{%
Viterbi Faculty of Electrical Engineering, Technion} \and Amarjit Budhiraja%
\thanks{%
Department of Statistics and Operations Research, University of North
Carolina} \and Paul Dupuis\thanks{%
Division of Applied Mathematics, Brown University} \and Ruoyu Wu\thanks{%
Department of Mathematics, University of Michigan}}
\maketitle

\begin{abstract}
\noindent For the M/M/1+M model at the law-of-large-numbers scale, the long
run reneging count per unit time does not depend on the individual (i.e.,
per customer) reneging rate. This paradoxical statement has a simple proof.
Less obvious is a large deviations analogue of this fact, stated as follows: 
\emph{The decay rate of the probability that the long run reneging count per
unit time is atypically large or atypically small does not depend on the
individual reneging rate.} In this paper, the sample path large deviations
principle for the model is proved and the rate function is computed. Next,
large time asymptotics for the reneging rate are studied for the case when
the arrival rate exceeds the service rate. The key ingredient is a calculus
of variations analysis of the variational problem associated with atypical
reneging. A characterization of the aforementioned decay rate, given
explicitly in terms of the arrival and service rate parameters of the model,
is provided yielding a precise mathematical description of this paradoxical
behavior. \newline
\ \newline

\noindent

\noindent \textbf{AMS subject classifications:} 60F10, 60J27, 60K25.\newline
\ \newline

\noindent \textbf{Keywords:} Single server queue, reneging, sample path
large deviations, Laplace principle, Euler-Lagrange equations, the reneging
paradox
\end{abstract}

\section{Introduction.}

\label{sec1}

Despite vast interest in recent years in queueing models with reneging and
their asymptotic analysis (see the survey article \cite{ward2012}), large
deviations (LD) treatment of even the simplest queueing model accounting for
reneging, namely the M/M/1+M, is lacking. This paper addresses two LD
aspects of this model: a sample path large deviations principle (LDP), and
the decay rate of the probability that the long run reneging count per unit
time is atypically large or atypically small when the arrival rate exceeds
the service rate. Theorem \ref{thm:I*} gives an explicit formula for the
aforementioned decay rate which shows in particular that the decay rate does
not depend on the parameter governing the individual (or per customer)
reneging rate.
An additional fact that arises from the analysis is that under the optimal change of measure associated with this atypical behavior, the number of reneging events over a large time interval, normalized by the cumulative time customers spend in the queue (summed over these customers), does not change w.r.t. to its law of large numbers (LLN) value. As a result, the LD cost associated with reneging vanishes.
These two phenomena, that are related to one
another, are called in this paper \textit{the reneging paradox at the LD
scale}.

The model under consideration is of a single server queue with reneging, in
which the interarrival times, service times and patience are exponentially
distributed. With $n\in{\mathbb{N}}$ as a scaling parameter, the arrival
rate in the $n$th system is given by $\lambda n$, the service rate by $\mu n$%
, and the per-customer reneging rate is given by $\theta$. That is, the
patience of each customer is an exponential random variable with parameter $%
\theta$. This scaling, in which arrival and service times are accelerated
but patience remains constant, is common in the literature on scaling limits
of queueing systems with reneging (as, for example, in \cite{atakasshi}, 
\cite{kanram10}). The first main result of this paper is Theorem \ref%
{th:thm1} which gives the sample path LDP for the process consisting of the
pair: normalized queue length, normalized reneging count, where
normalization refers to dividing by $n$. The technique for establishing this
is based on describing the state dynamics by Poisson random measures (PRM)
and proceeds by proving the Laplace upper and lower bounds using a general
variational representation for expectations of functionals of PRM from \cite%
{BudhirajaDupuisMaroulas2011variational}.

Simple considerations based on the balance equation identify the reneging
rate in equilibrium at the LLN scale. That is, the
number of arrivals over a time window is given by the number of departures
plus the number of reneging customers over this window with a correction
term that accounts for changes in the queue length. In steady state, this
correction term converges in probability to zero at the scaling limit.
Moreover, if $\lambda>\mu$ then in steady state the queue rarely becomes
empty in this asymptotic regime, and consequently the departure rate is well
approximated by $n\mu$, yielding reneging rate $\lambda n-\mu n$, with $o(n)$
correction. For the normalized processes, defined by division by $n$, this
rate is given by $\lambda-\mu$. On the other hand, if $\lambda\le\mu$ then
it is not hard to see that the reneging rate is $O(1)$, and therefore its
normalized version is asymptotic to zero. The lack of dependence of the
reneging rate on the parameter $\theta$ is also suggested by the time
asymptotic behavior of the fluid limit equations given in %
\eqref{eq:fluidlimit}. This gives a paradoxical behavior (that one may call 
\textit{the reneging paradox at the LLN scale}) that the overall reneging
rate is asymptotically independent of the per-customer reneging rate. It is
not hard to establish this property rigorously.
{
A control theoretic version of this phenomenon has been observed earlier
in \cite{puh19}, where a problem of minimizing
the long run average reneging cost of a fluid model was studied.
An optimal control policy was calculated, and it was
noticed that this policy is independent of the reneging distribution
(see \cite[Section 6]{puh19}).
}

In this work we establish a similar property at the large deviations scale.
Specifically, we are concerned with estimating the probabilities of
atypically large or small reneging count over a long time interval. 
We note that the reneging count is an important performance measure for queuing systems and probability of non-typical reneging counts, just as 
the buffer overflow probabilities for finite buffer queuing systems, are natural to analyze using a large deviation scaling.
If $%
R^{n}(t)$ denotes the reneging count at time $t$ in the $n$th system, and $%
\bar{R}^{n}(t)=n^{-1}R^{n}(t)$ denotes its normalized version, then by the
discussion above, $\bar{R}^{n}(t)$ converges in probability to $\gamma _{0}t$ as $%
n\rightarrow \infty $, where $\gamma _{0}=(\lambda -\mu )^{+}$. Then the
quantities of interest are, for $\gamma >\gamma _{0}$, 
\begin{equation*}
\chi ^{+}(\gamma )=\limsup_{t\rightarrow \infty }\limsup_{n\rightarrow
\infty }\frac{1}{t}\frac{1}{n}\log P(\bar{R}^{n}(t)\geq \gamma t),\qquad
\chi ^{-}(\gamma )=\liminf_{t\rightarrow \infty }\liminf_{n\rightarrow
\infty }\frac{1}{t}\frac{1}{n}\log P(\bar{R}^{n}(t)>\gamma t),
\end{equation*}%
and for $\gamma \in (0,\gamma _{0})$ (in the case $\gamma _{0}>0$), 
\begin{equation*}
\chi ^{+}(\gamma )=\limsup_{t\rightarrow \infty }\limsup_{n\rightarrow
\infty }\frac{1}{t}\frac{1}{n}\log P(\bar{R}^{n}(t)\leq \gamma t),\qquad
\chi ^{-}(\gamma )=\liminf_{t\rightarrow \infty }\liminf_{n\rightarrow
\infty }\frac{1}{t}\frac{1}{n}\log P(\bar{R}^{n}(t)<\gamma t).
\end{equation*}%
Under the assumption $\lambda \geq \mu $, Theorem \ref{thm:I*} (see also
Remark \ref{rem:limsupinf}) shows that $\chi ^{+}(\gamma )=\chi ^{-}(\gamma
) $ and provides a formula for this quantity as an explicit function of $%
\lambda ,\mu $ and $\gamma $. The tools are those of the calculus of
variations. That is, the variational formula provided by the LDP for the
large $n$ asymptotics is analyzed for each $t$ via the Euler-Lagrange
equations. This analysis gives an expression for the minimizing trajectories
in this variational problem, for fixed $t$ that is sufficiently large, which
is explicit, except that it involves one scalar parameter $A$. This
parameter is characterized as the solution of a certain nonlinear equation
(see \eqref{eq:A-def} and Lemma \ref{lem:minimizer-existence}) which for a
fixed $t$ does not admit a simple form solution. We study properties of this
parameter as a function of the initial condition and the time horizon as
this time horizon approaches infinity. Using these properties we then
analyze the scaled optimal cost in the variational problem as $t\rightarrow
\infty $ and obtain a simple form expression for the limit as a function of $%
\lambda ,\mu $ and $\gamma $. Finally it is argued using several nice
properties of the rate function that this limit quantity equals $\chi
^{+}(\gamma )$ and $\chi ^{-}(\gamma )$. In this work we do not consider the
case $\lambda <\mu $. The relevant calculus of variations problem for this
case is less tractable and its study will be taken up elsewhere.

We now present a heuristic to justify the formula for the decay rate 
as well as the two aspects of the reneging paradox at the LD scale. Consider
for specificity, the event that the normalized reneging count over an
interval of time $T$ exceeds $\gamma T$, where $\gamma>\gamma_0$. As is well
known, a Poisson process with rate $\alpha$ satisfies a sample path LDP in $%
\mathcal{D}([0,T]:{\mathbb{R}}_+)$, with rate function given by $%
I(\varphi)=\int_0^T \alpha\ell(\dot\varphi(s)/\alpha)ds$ for absolutely
continuous functions $\varphi:[0,T]\to{\mathbb{R}}_+$, and $%
I(\varphi)=\infty $ all other elements of $\mathcal{D}([0,T]:{\mathbb{R}}_+)$. Here, $\mathcal{D}([0,T]:{\mathbb{R}}_+)$ is 
the space of right continuous functions with left limits from $[0,T]$ to ${\mathbb{R}}_+$ equipped with the Skorohod topology, and $\ell(u)=u\log u-u+1$ for $u>0$, $\ell(0)=1$ and $\ell(u)=\infty$
for $u<0$. The function $\ell$ attains its minimum value $0$ uniquely at $1$. In view of this, and using balance equation considerations at equilibrium
similar to those described for the LLN analysis, one may conjecture the
following. The decay rate of the probability of this event as $n\to\infty$
is given by 
\begin{equation}  \label{102}
-T\inf\{\lambda\ell(\lambda^*/\lambda)+\mu\ell(\mu^*/\mu)+\theta
y\ell(\theta^*/\theta)\},
\end{equation}
with an $o(T)$ correction, where the infimum is over $\lambda^*,\mu^*,%
\theta^*$ and $y$ in $(0,\infty)$ satisfying 
\begin{equation*}
\lambda^* = \mu^* +\theta^* y, \qquad \theta^* y = \gamma.
\end{equation*}
Moreover, the minimizing $(\lambda^*,\mu^*,\theta^*)$ correspond to the
parameters of the optimal change of measure in the LD analysis. To solve
this optimization problem, note that the first constraint can be written as $%
\lambda^*=\mu^*+\gamma$, which decouples the problem into two optimization
problems, one associated with $(\lambda^*,\mu^*)$, another with $%
(\theta^*,y)$. Clearly, the latter is solved by $\theta^*=\theta$, which
makes the last term vanish. This heuristic supports the statement that no
change of measure is associated with the reneging process under the optimal
change of measure governing the event of interest, as well as the fact that
the solution is independent of $\theta$. In addition, it can be checked that
the solution of \eqref{102} is $-TC(\gamma)$ where $C(\gamma)$ is defined in %
\eqref{cgamma}. The above heuristic is made rigorous in Theorem \ref{thm:I*}.

Whereas this work focuses on a single server model, very similar results can
be established for the multi-server queue with reneging, such as in a
setting where the number of servers grows linearly with the scaling
parameter $n$, and each server operates with standard exponential service
time distribution (namely, the M/M/$n$+M system). In Section \ref%
{sec:multiser} we present and discuss these results, however the proofs are
omitted. The specific scaling used in this setting has been referred to as a 
\textit{many-server} scaling. This model and scaling were studied for their
LLN and CLT asymptotics in \cite{kasram11}, \cite{kasram13}, \cite{reed09}
far beyond the exponential setting. 

Our main motivation for this work stems from an approach to obtain robust
probability estimates at the LD scale. According to this approach, a family
of probabilistic models is specified in terms of R\'enyi divergence w.r.t.\
a single reference model (that is often easier to analyze). The approach
then provides a tool by which LD estimates on the reference model can be
transformed into LD estimates on the whole family. The model G/G/$n$+G is a
particularly good test case for this approach because (a) the model is hard,
and (b) robustness w.r.t.\ underlying distributions is important in
applications where it is often used, such as in models for call centers. The
estimates obtained in this paper for the Markovian model can be used by this
approach to obtain LD estimates for suitable families of G/G/$n$+G in terms
of the results of this paper concerning M/M/$n$+M. This progress will be
reported in \cite{ABDW2}.

\subsection{Notation.}

Let $\mathcal{D}([0,T]: {\mathbb{R}}^d)$ (resp.\ $\mathcal{C}([0,T]: {\mathbb{%
R}}^d)$) denote the space of right continuous functions with left limits
(resp.\ continuous functions) from $[0,T]$ to ${\mathbb{R}}^d$, equipped with
the usual Skorohod (resp.\ uniform) topology. Define $|x|_{*,T} \doteq
\sup_{0 \le t \le T} |x(t)|$ for $x \in {\mathcal{D}}([0,T]:{\mathbb{R}}^d)$%
. For a Polish space $\mathcal{E}$, we denote by $C_b(\mathcal{E})$ the
space of real continuous and bounded functions on $\mathcal{E}$ and by $\mathcal{P}(%
\mathcal{E})$ the space of probability measures on $\mathcal{E}$ which is
equipped with the topology of weak convergence. We say a collection of $%
\mathcal{E}$ valued random variables is tight if the corresponding family of
probability laws of the random variables is relatively compact in the space $%
\mathcal{P}(\mathcal{E})$. 
{A function $I:\mathcal{E} \to [0, \infty]$ will be called a rate function if for every $m<\infty$, 
the set $\{x: I(x)\le m\}$ is compact. Some authors also refer to such a function as a `good rate function', however here we will drop the adjective `good'.}
A tight sequence of $\mathcal{D}([0,T]: {\mathbb{R%
}}^d)$ valued random variables is said to be ${\mathcal{C}}$-tight if the
limit of every weakly convergent subsequence takes values in $\mathcal{C}%
([0,T]: {\mathbb{R}}^d)$ a.s. We will use $\kappa$, $\kappa_1$, $\kappa_2$, $%
\dotsc$ to denote constants in the proof.

\section{Large Deviations for $M/M/1$ with Reneging.}

\label{sec:setting} We begin by describing the evolution of the scaled state
process. For this it will be convenient to represent the jumps in the system
through certain Poisson random measures, which are introduced below.

For a locally compact Polish space $\mathbb{S}$, let $\mathcal{M}_{FC}(%
\mathbb{S})$ be the space of all measures $\nu$ on $(\mathbb{S},\mathcal{B}(%
\mathbb{S}))$ such that $\nu(K)<\infty$ for every compact $K \subset \mathbb{%
S}$. We equip $\mathcal{M}_{FC}(\mathbb{S})$ with the usual vague topology.
This topology can be metrized such that $\mathcal{M}_{FC}(\mathbb{S})$ is a
Polish space (see \cite{BudhirajaDupuisMaroulas2011variational} for one
convenient metric). A Poisson random measure (PRM) $N$ on a locally compact
Polish space $\mathbb{S}$ with mean measure (or intensity measure) $\nu \in 
\mathcal{M}_{FC}(\mathbb{S})$ is an $\mathcal{M}_{FC}(\mathbb{S})$ valued
random variable such that for each $A \in \mathcal{B}(\mathbb{S})$ with $%
\nu(A)<\infty$, $N(A)$ is Poisson distributed with mean $\nu(A)$ and for
disjoint $A_1,\dotsc,A_k \in \mathcal{B}(\mathbb{S})$, $N(A_1),\dotsc,N(A_k)$
are mutually independent random variables (cf.\ \cite{IkedaWatanabe1990SDE}).

Fix $T\in (0,\infty )$. Let $(\Omega ,\mathcal{F},{P})$ be a complete
probability space on which we are given three mutually independent PRM $%
N_{1},N_{2},N_{3}$ on $[0,T]\times {\mathbb{R}}_{+}$, $[0,T]\times {\mathbb{R%
}}_{+}$ and $[0,T]\times {\mathbb{R}}_{+}^{2}$ respectively with intensities 
$\lambda \,ds\times dy$, $\mu \,ds\times dy$ and $\theta \,ds\times dy\times
dz$ respectively. Let ${\mathbb{X}}_{1}\doteq {\mathbb{R}}_{+}$, ${\mathbb{X}%
}_{2}\doteq {\mathbb{R}}_{+}$, ${\mathbb{X}}_{3}\doteq {\mathbb{R}}_{+}^{2}$%
. Define the filtration 
\begin{equation*}
\hat{\mathcal{F}}_{t}\doteq \sigma \{N_{i}((0,s]\times A_{i}),s\in \lbrack
0,t],A_{i}\in \mathcal{B}({\mathbb{X}}_{i}),i=1,2,3\},\:t\geq 0
\end{equation*}%
and let $\{\mathcal{F}_{t}\}$ be the ${P}$-augmentation of this filtration.
Let $\mathcal{\bar{P}}$ be the $\{\mathcal{F}_{t}\}_{0\leq t\leq T}$%
-predictable $\sigma $-field on $\Omega \times \lbrack 0,T]$. Denote by $%
\bar{\mathcal{A}}_{1},\bar{\mathcal{A}}_{3}$ the class of all measurable
maps from $(\Omega \times \lbrack 0,T],\mathcal{\bar{P}})$, $(\Omega \times
\lbrack 0,T]\times {\mathbb{R}}_{+},\mathcal{\bar{P}}\times \mathcal{B}(%
\mathbb{R}_{+}))$ to $(\mathbb{R}_{+},\mathcal{B}(\mathbb{R}_{+}))$,
respectively. Let $\bar{\mathcal{A}}_{2}\doteq \bar{\mathcal{A}}_{1}$.

For $\varphi \in \bar{\mathcal{A}}_i$, define {counting processes $N_1^{\varphi}$, $N_2^{\varphi}$, $N_3^{\varphi}$ on $[0,T]$, $[0,T]$, $[0,T]\times{\mathbb{R}}_+$, respectively,} by 
\begin{align*}
N_i^{\varphi}([0,t]) & \doteq \int_{[0,t]\times \mathbb{R}_{+}}{1}%
_{[0,\varphi(s)]}(y)\,N_i(ds\,dy),\:t\in [0,T], i=1,2, \\
N_i^{\varphi}([0,t]\times A) & \doteq \int_{[0,t]\times A\times \mathbb{R}%
_{+}}{1}_{[0,\varphi (s,y)]}(z)\,N_i(ds\,dy\,dz),\:t\in [0,T],A\in \mathcal{B%
}({\mathbb{R}}_+), i=3.
\end{align*}
We think of $N_i^{\varphi }$ as a controlled random measure, where $\varphi$
is the control process that produces a thinning of the point process $N_i$
in a random but non-anticipative manner to produce a desired intensity. We
will write $N_i^{\varphi}$ as $N_i^{m}$ if $\varphi \equiv m$ for some
constant $m \in \mathbb{R}_{+}$. Note that $N_1^m,N_2^m,N_3^m$ are PRM on $%
[0,T],[0,T],[0,T]\times {\mathbb{R}}_+$ with intensity $m\lambda \, ds$, $%
m\mu \, ds$, $m\theta \, ds \times dy$ respectively.

\subsection{State dynamics using Poisson random measures.}

{
The process representing the number of customers in the $n$th system
will be referred to as the {\it number in system} process, denoted by
$U^n(t)$. Note that the queue length process and the number of customers
in service process (the latter being $\{0,1\}$-valued) can be expressed
as $(U^n(t)-1)^+$ and $U^n(t)\w1$, respectively.
Denote the total reneging by time  $t$ by $V^{n}(t)$.
Let rescaled versions of these processes be defined by
\begin{equation*}
X^{n}(t)=\frac{U^{n}(t)}{n},\;Y^{n}(t)=\frac{V^{n}(t)}{n},\;t\in \lbrack
0,T].
\end{equation*}%
}
We take $(X^{n}(0),Y^{n}(0))=(x_{n},0)$ with $x_{n}\rightarrow x_{0}$ as $%
n\rightarrow \infty $. We will establish a LDP for $(X^{n},Y^{n})$ in $%
\mathcal{D}([0,T]:{\mathbb{R}}_{+}^{2})$ and then deduce the LDP for $%
Y^{n}(T)$ using the contraction principle.

Using the PRMs introduced above, the state evolution can be written as 
\begin{align*}
X^{n}(t)& =x_{n}+\frac{1}{n}\int_{[0,t]}N_{1}^{n}(ds)-\frac{1}{n}\int_{[0,t]}%
{\boldsymbol{1}}_{\{X^{n}(s-)\neq 0\}}\,N_{2}^{n}(ds) \\
& \quad -\frac{1}{n}\int_{[0,t]\times {\mathbb{R}}_{+}}{\boldsymbol{1}}%
_{\left[0,\left( X^{n}(s-)-\frac{1}{n}\right) ^{+}\right]}(y)\,N_{3}^{n}(ds\,dy). \\
Y^{n}(t)& =\frac{1}{n}\int_{[0,t]\times {\mathbb{R}}_{+}}{\boldsymbol{1}}%
_{\left[0,\left( X^{n}(s-)-\frac{1}{n}\right) ^{+}\right]}(y)\,N_{3}^{n}(ds\,dy).
\end{align*}%
{
The role of the indicator of $\{X^n(s-)\ne0\}$ in the second integral on the RHS
is to express the fact that departures occur only when the number in system
is non-zero. Moreover, notice that the expression $(X^n-\frac{1}{n})^+$
gives the normalized queue length. Hence to model reneging according
to exponential clocks for each customer in the queue,
the third integral expresses
reneging that occurs at a rate proportional to the (normalized) queue length.
}

Define the map $\Gamma :\mathcal{D}([0,T]:{\mathbb{R}})\rightarrow \mathcal{D%
}([0,T]:{\mathbb{R}}_{+})$ by 
\begin{equation}
\Gamma (\psi )(t)\doteq \psi (t)-\inf_{0\leq s\leq t}[\psi (s)\wedge
0],\quad t\in \lbrack 0,T],\psi \in \mathcal{D}([0,T]:{\mathbb{R}}).
\label{eq:SM}
\end{equation}%
The map $\Gamma $ is usually referred to as the one-dimensional Skorohod map
(see, e.g., \cite[Section 3.6.C]{KaratzasShreve1991brownian}).
{Note that the evolution of $X^n$ can be written as
\begin{align*}
X^{n}(t)& =x_{n}+\frac{1}{n}\int_{[0,t]}N_{1}^{n}(ds)-\frac{1}{n}\int_{[0,t]}%
\,N_{2}^{n}(ds) \\
& \quad -\frac{1}{n}\int_{[0,t]\times {\mathbb{R}}_{+}}{\boldsymbol{1}}%
_{\left[0,\left( X^{n}(s-)-\frac{1}{n}\right) ^{+}\right]}(y)\,N_{3}^{n}(ds\,dy) + \eta^n(t), 
\end{align*}
where $\eta^n(t) \doteq \frac{1}{n}\int_{[0,t]}%
{\boldsymbol{1}}_{\{X^{n}(s-) = 0\}}\,N_{2}^{n}(ds)$ increases only when $X^n(t-) = X^n(t)=0$.
Then using a well known characterization of the  Skorohod map (cf. \cite{dupish})
one can write the evolution of $X^{n}$ as:}
\begin{equation}\label{eq:rev1}
X^{n}(t)=\Gamma \left( x_{n}+\frac{1}{n}\int_{[0,\cdot ]}N_{1}^{n}(ds)-\frac{%
1}{n}\int_{[0,\cdot ]}N_{2}^{n}(ds)-\frac{1}{n}\int_{[0,\cdot ]\times {%
\mathbb{R}}_{+}}{\boldsymbol{1}}_{\left[0,\left( X^{n}(s-)-\frac{1}{n}\right) ^{+}\right]}(y)\,N_{3}^{n}(ds\,dy)%
\right) (t).
\end{equation}%
As $n\rightarrow \infty $, $(X^{n},Y^{n})$ converges in $\mathcal{D}([0,T]:{%
\mathbb{R}}_{+}^{2})$, in probability, to $(x,y)$ given by 
\begin{equation}
\begin{aligned} x(t) &= \Gamma \left(x_0 + (\lambda-\mu)\iota(\cdot) -
\theta \int_0^{\cdot} x(s)\,ds\right)(t), \; t \in [0,T], \\ y(t) &= \theta
\int_0^t x(s)\,ds, \; t \in [0,T], \end{aligned}  \label{eq:fluidlimit}
\end{equation}%
where $\iota $ is the identity map on $[0,T]$. Theorem \ref{th:thm1} will
establish a large deviation principle for $(X^{n},Y^{n})$ as $n\rightarrow
\infty $. We begin by introducing the associated rate function.

\subsection{Rate function and the large deviation principle.}\label{sec22}

For $(\xi, \zeta) \in \mathcal{C}([0,T]: {\mathbb{R}}_+^2)$, let $\mathcal{U}%
(\xi, \zeta)$ be the collection of all $\varphi = (\varphi_1, \varphi_2,
\varphi_3)$ such that $\varphi_i: [0,T] \to {\mathbb{R}}_+$, $i=1,2,3$ are
measurable maps and the following equations are satisfied for all $t \in
[0,T]$: 
\begin{equation}  \label{eq:eqxizet}
\begin{aligned} \xi(t) &= \Gamma\left (x_0 + \lambda\int_0^{\cdot}
\varphi_1(s)\,ds - \mu \int_0^{\cdot} \varphi_2(s)\,ds - \theta
\int_0^{\cdot} \varphi_3(s) \xi(s)\,ds\right)(t), \\ \zeta(t) &= \theta
\int_0^t \varphi_3(s)\xi(s)\,ds. \end{aligned}
\end{equation}
Define 
\begin{equation}  \label{eq:ell}
\ell(x) \doteq x \log x -x +1, \quad x \ge 0.
\end{equation}
For $(\xi, \zeta) \in \mathcal{C}([0,T]: {\mathbb{R}}_+^2)$, define 
\begin{equation}  \label{eq:I_T}
I_T(\xi, \zeta) \doteq \inf_{\varphi \in \mathcal{U}(\xi, \zeta)}\left\{
\lambda\int_0^T \ell(\varphi_1(s))\,ds + \mu\int_0^T \ell(\varphi_2(s))\,ds
+ \theta\int_0^T \xi(s)\ell(\varphi_3(s))\,ds\right\}.
\end{equation}
We set $I_T(\xi, \zeta)$ to be $\infty$ if $\mathcal{U}(\xi, \zeta)$ is
empty or $(\xi, \zeta) \in \mathcal{D}([0,T]: {\mathbb{R}}_+^2)\setminus 
\mathcal{C}([0,T]: {\mathbb{R}}_+^2)$.

\begin{theorem}
\label{th:thm1} The pair $\{(X^n, Y^n)\}$ satisfies a LDP on $\mathcal{D}%
([0,T]: {\mathbb{R}}_+^2)$ with rate function $I_T$.
\end{theorem}

Using the contraction principle { (cf. \cite[Theorem 1.3.2]{dupell4})} one has the following result. Let, for $%
\gamma \in {\mathbb{R}}_{+}$, $\mathcal{U}^{\ast }(\gamma )$ be the
collection of all $(\xi ,\zeta )\in \mathcal{C}([0,T]:{\mathbb{R}}_{+}^{2})$
such that $\zeta (T)=\gamma $. %
%
%
Let $I_{T}^{\ast }:{\mathbb{R}}_{+}\rightarrow \lbrack 0,\infty ]$ be
defined as 
\begin{equation}
I_{T}^{\ast }(\gamma )\doteq \inf_{(\xi ,\zeta )\in \mathcal{U}^{\ast
}(\gamma )}I_{T}(\xi ,\zeta ).  \label{eq:istart}
\end{equation}

\begin{theorem}
\label{thm:contra} $\{Y^n(T)\}$ satisfies a LDP on ${\mathbb{R}}_+$ with
rate function $I^*_T$.
\end{theorem}

There does not appear to be a simple form expression for $I^*_T(\gamma)$ for
fixed $T>0$ and $\gamma \ge 0$ and therefore we consider asymptotics of $%
I^*_T(\gamma)$ for large $T$. For this we restrict attention to the case $%
\lambda \ge \mu$. Let 
\begin{equation}  \label{eq:eqiii}
\begin{aligned} I_{\gamma, T}^* & \doteq I^*_T(\gamma T) = \inf \{
I_T(\xi,\zeta) : \zeta(T)=\gamma T\}, \quad \gamma \ge 0, \\ I_{\gamma,T} &
\doteq \inf \{ I_T(\xi,\zeta) : \zeta(T) \ge \gamma T \}, \quad \gamma \ge
\lambda - \mu, \\ {\tilde{I}}_{\gamma,T} & \doteq \inf \{ I_T(\xi,\zeta) :
\zeta(T) \le \gamma T \}, 0 \le \gamma \le \lambda-\mu. \end{aligned}
\end{equation}
Let 
\begin{equation}  \label{cgamma}
C(\gamma) \doteq \lambda \left(1-z_\gamma^{-1}\right) + \mu
\left(1-z_\gamma\right) - \gamma \log z_\gamma,
\end{equation}
where 
\begin{equation}  \label{eq:z}
z_\gamma \equiv z(\gamma) \doteq \frac{\sqrt{\gamma^2 + 4\lambda\mu}-\gamma}{%
2\mu} = \frac{2\lambda}{\sqrt{\gamma^2 + 4\lambda\mu}+\gamma}.
\end{equation}

\begin{theorem}
\label{thm:I*} Suppose that $\lambda \ge \mu$. We have the following
asymptotic formulae. 
\begin{align*}
\lim_{T \to \infty} \frac{I_{\gamma, T}^*}{T} & = C(\gamma), \quad \gamma
\ge 0, \\
\lim_{T \to \infty} \frac{I_{\gamma,T}}{T} & = C(\gamma), \quad \gamma \ge
\lambda-\mu, \\
\lim_{T \to \infty} \frac{{\tilde{I}}_{\gamma,T}}{T} & = C(\gamma), \quad 0
\le \gamma \le \lambda-\mu.
\end{align*}
\end{theorem}

\begin{remark}
\label{rem:limsupinf} From Theorem \ref{thm:I*} and the convexity and
monotonicity properties of $I_{\gamma, T}^*$ given in Lemma \ref%
{lem:IT_convex} it follows that the quantities $\chi^+(\gamma)$ and $%
\chi^-(\gamma)$ defined in the Introduction equal $C(\gamma)$ for every $%
\gamma>0$.
\end{remark}

{
To give an interpretation for taking both $n$ and $T$ to be large parameters,
and to demonstrate a simple use of our results,
consider a service capacity allocation problem.
In this system, the arrival rate is large whereas the reneging
rate per customer is order 1 (and unknown). We take $\la=1$, and $\theta>0$,
so that in the $n$th system the arrival rate is $n$,
and the individual reneging rate is $\theta$.
We wish to find the smallest service capacity parameter $\mu\in(0,1)$
(corresponding in the $n$th system to service at rate $\mu n$)
so that it is guaranteed that the reneging count $R^n(T)$ over a large interval
of time $[0,T]$, satisfies
$P(R^n(T)>\gamma nT)\le e^{-cnT}$ for given constants $\gamma$
and $c$. Consider for example $\gamma=1$ and {$c=0.1$}. Then
we have $z(\gamma) = \frac{\sqrt{1+4\mu}-1}{2\mu} = \frac{2}{\sqrt{1+4\mu}+1}$ and $C(\gamma) = 1+\mu - \sqrt{1+4\mu} - \log \frac{2}{\sqrt{1+4\mu}+1}$. Setting $C(\gamma)={0.1}$ gives $\mu \approx {0.5723}$. Note that the calculation
does not depend on $\theta$.
}

\section{Proof of Theorem \protect\ref{th:thm1}.}

Let $Z^n = (X^n, Y^n)$ and $\mathcal{E} = \mathcal{D}([0,T]: {\mathbb{R}}%
_+^2)$.  From the equivalence betweeen a large deviation principle and a Laplace principle \cite[Section 1.2]{dupell4} it suffices to
show that the function $I_T$ has compact sublevel sets and the following bounds hold for all $h \in C_b(\mathcal{E})$. \newline

\noindent \textbf{Laplace Upper Bound} 
\begin{equation}  \label{eq:lapuppaband}
\liminf_{n\to \infty} -\frac{1}{n} \log E \exp \left\{-n h(Z^n)\right\} \ge
\inf_{\phi \in \mathcal{E}}\left[I_T(\phi)+ h(\phi)\right].
\end{equation}
\noindent \textbf{Laplace Lower Bound} 
\begin{equation}  \label{eq:laplowaband}
\limsup_{n\to \infty} -\frac{1}{n} \log E \exp \left\{-n h(Z^n)\right\} \le
\inf_{\phi \in \mathcal{E}}\left[I_T(\phi)+ h(\phi)\right].
\end{equation}
The proof of compactness of level sets is analogous to the proof of the upper bound and is therefore omitted. The Laplace upper bound is proved in Section \ref{sec:lapuppaband}
and the lower bound is established in Section \ref{sec:laplowbd}.

\subsection{Variational Representation and Weak Convergence of Controlled Processes.}

We will use the following useful representation formula proved in \cite%
{BudhirajaDupuisMaroulas2011variational} and \cite[Theorem 2.4]%
{BudhirajaChenDupuis2013large} (see also \cite[Theorem 8.12]{BudhirajaDupuis2019analysis}). Recall that $\bar{\mathcal{A}}_{1},\bar{%
\mathcal{A}}_{2},\bar{\mathcal{A}}_{3}$ denote the class of all measurable
maps from $(\Omega \times \lbrack 0,T],\mathcal{\bar{P}})$, $(\Omega \times
\lbrack 0,T],\mathcal{\bar{P}})$, $(\Omega \times \lbrack 0,T]\times {%
\mathbb{R}}_{+},\mathcal{\bar{P}}\times \mathcal{B}(\mathbb{R}_{+}))$ to $(%
\mathbb{R}_{+},\mathcal{B}(\mathbb{R}_{+}))$, respectively. For each $m\in 
\mathbb{N}$, let 
\begin{align*}
& \bar{\mathcal{A}}_{b,k}\doteq \{(\varphi _{1},\varphi _{2},\varphi
_{3}):\varphi _{i}\in \bar{\mathcal{A}}_{i}\mbox{ for each }i=1,2,3%
\mbox{
such that for all }(\omega ,t,y)\in \Omega \times \lbrack 0,T]\times {%
\mathbb{R}}_{+}, \\
& \qquad \qquad \qquad \qquad \qquad \varphi _{1}(\omega ,t),\varphi
_{2}(\omega ,t),\varphi _{3}(\omega ,t,y)\in \lbrack 1/k,k], {\text{ and } \varphi _{3}(\omega ,t,y)=1 \text{ when } y > k}\}
\end{align*}%
and let $\bar{\mathcal{A}}_{b}\doteq \cup _{k=1}^{\infty }\bar{\mathcal{A}}%
_{b,k}$.
{
(With a slight abuse of notation, here and in what follows,
$\varphi_i$ denote stochastic processes, unlike
in Section \ref{sec22} where $\varphi_i$ were used to denote deterministic
functions).
}
Denote ${\mathcal{M}}\doteq \lbrack \mathcal{M}_{FC}([0,T])]^{2}%
\times \mathcal{M}_{FC}([0,T]\times {\mathbb{R}}_{+})$. Recall the function $%
\ell $ defined in \eqref{eq:ell}. Let for $m\in {\mathbb{N}}$, $%
N^{m}=(N_{1}^{m},N_{2}^{m},N_{3}^{m})$ and for $\varphi =(\varphi
_{1},\varphi _{2},\varphi _{3})\in \times _{i=1}^{3}\bar{\mathcal{A}}_{i}$,
let $N^{\varphi }=(N_{1}^{\varphi _{1}},N_{2}^{\varphi _{2}},N_{3}^{\varphi
_{3}})$. $N^{m}$, $N^{\varphi }$ are regarded as a ${\mathcal{M}}$ valued
random variables. The theorem represents an expected value in terms of
infima over both $\times _{i=1}^{3}\bar{\mathcal{A}}_{i}$ and $\bar{\mathcal{%
A}}_{b}$. The latter is sometimes more convenient since for each fixed
control there are uniform upper and lower (away from zero) bounds.

\begin{theorem}
\label{thm:var_repn} Let $F$ be a real valued bounded measurable map on  ${\mathcal{M}}$. Then for $m >
0 $, 
\begin{align*}
& -\log {E} e^{-F(N^m)} \\
& \quad = \inf_{\varphi = (\varphi_i)_{i=1,2,3}\in \times_{i=1}^3\bar{%
\mathcal{A}}_i}{E} \left[ m \int_0^T \left( \lambda \ell(\varphi_1(s)) + \mu
\ell(\varphi_2(s)) + \theta \int_{{\mathbb{R}}_+} \ell(\varphi_3(s,y))\,dy
\right)ds + F(N^{m\varphi}) \right] \\
& \quad = \inf_{\varphi = (\varphi_i)_{i=1,2,3} \in \bar{\mathcal{A}}_b} {E} %
\left[ m \int_0^T \left( \lambda \ell(\varphi_1(s)) + \mu \ell(\varphi_2(s))
+ \theta \int_{{\mathbb{R}}_+} \ell(\varphi_3(s,y))\,dy \right)ds +
F(N^{m\varphi}) \right].
\end{align*}
\end{theorem}
{We note that the representation  in \cite{BudhirajaDupuisMaroulas2011variational,
BudhirajaChenDupuis2013large, BudhirajaDupuis2019analysis} is given for a single PRM with points in $[0,T]\times \mathcal{X}$
where $\mathcal{X}$ is some locally compact space. One can identify the triplet $N^m = (N^m_1, N^m_2, N^m_3)$ with a single PRM with points in the space $[0,T] \times \Rmb_+ \times \{1,2,3\}$ and intensity $\leb_T \times\nu $
where, $\leb_T$ is the Lebesgue measure on $[0,T]$ and $\nu$ is a locally-finite measure on $\mathcal{X} =  \Rmb_+ \times \{1,2,3\}$ defined as
{
\begin{align*}
	& \nu(\{0\}\times \{1\}) =m\lambda, \quad \nu((0,\infty)\times\{1\} ) = 0, \\
	& \nu(\{0\}\times \{2\}) =m\mu, \quad \nu((0,\infty)\times\{2\} ) = 0, \\
	& \nu( A\times \{3\}) = m\theta \: \leb_{\infty}(A), \quad A \in \Bmc(\Rmb_+),
\end{align*}
}where $\leb_\infty$ is the Lebesgue measure on $\Rmb_+$. With this identification the proof of Theorem \ref{thm:var_repn}
is immediate from the above references. In particular the proof of the fact that one can restrict the infimum to the smaller class $\bar{\mathcal{A}}_b$ can be found in \cite[Theorem 8.12]{BudhirajaDupuis2019analysis}.}
\begin{remark}
	\label{rem:funcpsi}
We mention difficulties to apply the continuous mapping approach.

\noindent (i) Let  $\mathcal{M}^0_{FC}([0,T])$ [resp.\ $\mathcal{M}^0_{FC}([0,T]\times {\mathbb{R}}_{+})$] be the subset of $\mathcal{M}_{FC}([0,T])$ [resp.\ $\mathcal{M}_{FC}([0,T]\times {\mathbb{R}}_{+})$] consisting of all `atomic measures' $\nu$ of the form  such that there is some finite set $F \subset [0,T]$ with $\nu(F^c) = 0$ [resp.\ for every $k \in \Nmb$ there is a finite set $F_k$ with $\nu(F_k^c \times [0,k])=0$].
	Let $\Mmc^0 \doteq \lbrack \mathcal{M}_{FC}^0([0,T])]^{2}%
\times \mathcal{M}_{FC}^0([0,T]\times {\mathbb{R}}_{+})$.
	Then for any $\nu = (\nu_1, \nu_2, \nu_3) {\in \Mmc^0}$, there is a unique $z^{\nu}=(x^{\nu},y^{\nu}) \in \Emc$ that satisfies
	\begin{align*}
	x^{\nu}(t)& =x_{n}+\frac{1}{n}\nu_1[0,t]-\frac{1}{n}\int_{[0,t]}%
	{\boldsymbol{1}}_{\{x^{\nu}(s-)\neq 0\}}\,\nu_2(ds)  -\frac{1}{n}\int_{[0,t]\times {\mathbb{R}}_{+}}{\boldsymbol{1}}%
	_{\left[0,\left( x^{\nu}(s-)-\frac{1}{n}\right) ^{+}\right]}(y)\,\nu_3(ds\,dy) \\
	y^{\nu}(t)& =\frac{1}{n}\int_{[0,t]\times {\mathbb{R}}_{+}}{\boldsymbol{1}}%
	_{\left[0,\left( x^{\nu}(s-)-\frac{1}{n}\right) ^{+}\right]}(y)\,\nu_3(ds\,dy).
	\end{align*}%
	Define $\Psi^n(\nu) = x^{\nu}$ for $\nu \in \Mmc^0$. For $\nu \in \Mmc\setminus \Mmc^0$, the map is defined in an arbitrary manner so that $\Psi^n$ is a measurable map from $\Mmc$ to $\Emc$. We then have that $Z^n= \Psi^n(N^n)$ a.s. for each $n$.
	We note that the above system of equations may not be wellposed for a general $\nu \in \Mmc \setminus \Mmc^0$ and for this reason there is no natural way to define $\Psi^n$ on all of $\Mmc$ with nice regularity properties (e.g. continuity). In particular there is no obvious way to implement a contraction principle approach using this representation in order to prove the LDP. 

\noindent (ii)
It is also possible to write $X$ in
terms of Poisson processes on the line $M$, $N$, $K$
(of rates $\la$, $\mu$, $\theta$, respectively) as
\begin{equation*}
X=\Gamma \Big(x+M-N-K\Big(\int_{0}^{\cdot }X(s)ds\Big)\Big),
\end{equation*}%
in which case $X|_{[0,t]}$ is a function of $M|_{[0,t]}$, $N|_{[0,t]}$ and $%
K|_{[0,\infty )}$. Letting $M^{n}$, $N^{n}$ denote accelerated by $n$
versions (recall that reneging rate is not accelerated), and $X^{n}$ the
resulting queue length, specializing to zero initial condition, we have $%
X^{n}=\Gamma (M^{n}-N^{n}-K(\int_{0}^{\cdot }X^{n}(s)ds))$. We can then
write versions that are both accelerated and rescaled by letting $\bar{M}%
^{n}=n^{-1}M(n\cdot )$, $\bar{N}^{n}=n^{-1}N(n\cdot )$, $\bar{K}%
^{n}=n^{-1}K(n\cdot )$, $\bar{X}^{n}=n^{-1}X^{n}$, and thereby obtain $\bar{X%
}^{n}=\Gamma (\bar{M}^{n}-\bar{N}^{n}-\bar{K}^{n}(\int_{0}^{\cdot }\bar{X}%
^{n}(s)ds))$. From this we see that $\bar{X}^{n}|_{[0,t]}$ depends on $\bar{K}%
^{n}|_{[0,\infty )}$ and not just $\bar{K}^{n}|_{[0,t]}$, and therefore, again,
it is not
clear how one can address the problem via the continuous mapping approach.
	
\end{remark}

Fix $h \in C_b({\mathcal{E}})$. Since $Z^n$ can be written as $\Psi^n (N^n)$
for some measurable function $\Psi^n $ from ${\mathcal{M}}$ to ${\mathcal{E}}$ 
we have from the second equality in Theorem \ref{thm:var_repn} that with $%
(m,F)=(n,nh\circ \Psi^n)$, 
\begin{align}
-\frac{1}{n} \log E e^{-n h(Z^n)} & = \inf {E} \left[ \int_0^T \left(
\lambda \ell(\varphi_1^n(s)) + \mu \ell(\varphi_2^n(s)) + \theta \int_{{%
\mathbb{R}}_+} \ell(\varphi_3^n(s,y))\,dy \right)ds + h({\bar{Z}}^n) \right],
\label{eq:mainrepn}
\end{align}
where the infimum is taken over all $\varphi^n = (\varphi_i^n)_{i=1,2,3} \in 
\bar{\mathcal{A}}_b$ and
$\bar Z^n = (\bar X^n, \bar Y^n)$ solves 
\begin{align}
\bar X^n(t) &= x_n + \frac{1}{n} N_1^{n\varphi_1^n}(t) - \frac{1}{n}
\int_0^t {\boldsymbol{1}}_{\{{\bar{X}}^n(s-)\neq 0\}}
\,N_2^{n\varphi_2^n}(ds) \\
&\quad - \frac{1}{n} \int_{[0,t]\times {\mathbb{R}}_+} {\boldsymbol{1}}_{\left[0,
\left({\bar{X}}^n(s-) - \frac{1}{n}\right)^+\right]}(y)
\,N_3^{n\varphi_3^n}(ds\,dy)  \notag \\
& = \Gamma \left(x_n + \frac{1}{n} N_1^{n\varphi_1^n}(\cdot) - \frac{1}{n}
N_2^{n\varphi_2^n}(\cdot) - \frac{1}{n} \int_{[0,\cdot]\times {\mathbb{R}}%
_+} {\boldsymbol{1}}_{\left[0, \left({\bar{X}}^n(s-) - \frac{1}{n}\right)^+\right]}(y)
\, N_3^{n\varphi_3^n}(ds\,dy)\right) (t),  \label{eq:Xbar_n} \\
\bar Y^n(t) &= \frac{1}{n} \int_{[0,t]\times {\mathbb{R}}_+} {\boldsymbol{1}}%
_{\left[0, \left({\bar{X}}^n(s-) - \frac{1}{n}\right)^+\right]}(y) \,
N_3^{n\varphi_3^n}(ds\,dy).  \label{eq:Ybar_n}
\end{align}

In the proof of both the upper and lower bound (see below Lemma \ref{lem:iitil} and below Proposition \ref{prop:uniq}) it will be sufficient to
consider a sequence $\{\varphi ^{n}\}\subset \bar{\mathcal{A}}_{b}$ that
satisfies the following uniform bound for some $M_{0}<\infty $:
\begin{equation}
\sup_{n\in \mathbb{N}}\int_{0}^{T}\left( \ell (\varphi _{1}^{n}(s))+\ell
(\varphi _{2}^{n}(s))+\int_{{\mathbb{R}}_{+}}\ell (\varphi
_{3}^{n}(s,y))\,dy\right) ds\leq M_{0},\mbox{ a.s.\ }{P}.
\label{eq:cost_bd_repn}
\end{equation}%
In the rest of this section we study tightness and convergence properties of
controlled processes $\{{\bar{Z}}^{n}\}$ that are driven by controls $%
\{\varphi ^{n}=(\varphi _{1}^{n},\varphi _{2}^{n},\varphi _{3}^{n})\}$ that
satisfy the a.s.\ bound (\ref{eq:cost_bd_repn}).

For $0 \le K<\infty$, let ${\mathcal{S}}_K$ be the collection of all
triplets $g = (g_1, g_2, g_3)$, where $g_1, g_2 \colon [0,T] \to {\mathbb{R}}%
_+$, $g_3 \colon [0,T]\times {\mathbb{R}}_+ \to {\mathbb{R}}_+$ are
measurable maps such that 
\begin{equation*}
\int_0^T \left( \ell(g_1(s)) + \ell(g_2(s)) + \int_{{\mathbb{R}}_+}
\ell(g_3(s,y))\,dy \right) ds \le K.
\end{equation*}
An element $g = (g_1,g_2,g_3) \in {\mathcal{S}}_K$ can be identified with
elements $\nu^g \in {\mathcal{M}}$, where 
\begin{equation}  \label{eq:nu}
\nu^g_i(A) \doteq \int_A g_i(s)\,ds, \:\: \nu^g_3(A\times B) \doteq \int_{A
\times B} g_3(s,y) \,ds\,dy, \:\: i=1,2, A \in {\mathcal{B}}([0,T]), B \in {%
\mathcal{B}}({\mathbb{R}}_+).
\end{equation}
With this identification ${\mathcal{S}}_K$ is a compact metric space (cf.\ 
\cite[Lemma A.1]{BudhirajaChenDupuis2013large}). Let $\tilde{\mathcal{A}}%
_{b,K} \doteq \{ \varphi \in \bar{\mathcal{A}}_b: \varphi(\omega) \in {%
\mathcal{S}}_K \mbox{ a.s. }\}$.

For $\varphi^n \in \bar{\mathcal{A}}_b$, define the compensated processes 
\begin{align*}
{\tilde{N}}^{n\varphi_1^n}_1(ds) & \doteq N^{n\varphi_1^n}_1(ds) - \lambda n
\varphi_1^n(s)\,ds, \\
{\tilde{N}}^{n\varphi_2^n}_2(ds) & \doteq N^{n\varphi_2^n}_2(ds) - \mu n
\varphi_1^n(s)\,ds, \\
{\tilde{N}}^{n\varphi_3^n}_3(ds\,dy) & \doteq N^{n\varphi_3^n}_3(ds\,dy) -
\theta n \varphi_3^n(s,y)\,ds\,dy.
\end{align*}
Then ${\tilde{N}}^{n\varphi_1^n}_1([0,t])$, ${\tilde{N}}^{n%
\varphi_2^n}_2([0,t])$, ${\tilde{N}}^{n\varphi_3^n}_3([0,t] \times A)$ are $%
\{{\mathcal{F}}_t\}$-martingales for $A \in {\mathcal{B}}({\mathbb{R}}_+)$
with $E(\nu^{\varphi^n}_3([0,T]\times A))<\infty$.

The following lemma summarizes some elementary properties of $\ell$. For
part (a) we refer to \cite[Lemma 3.1]{BudhirajaDupuisGanguly2015moderate},
and part (b) is an easy calculation that is omitted.

\begin{lemma}
\phantomsection
\label{lem:property_ell}

\begin{enumerate}[(a)]

\item For each {$\beta > 1$}, there exists $\gamma(\beta) \in (0,\infty)$ such
that $\gamma(\beta) \to 0$ as $\beta \to \infty$ and $x \le \gamma(\beta)
\ell(x)$, for {$x \ge \beta$}.

\item For $x \ge 0$, $x \le \ell(x) + 2$.
\end{enumerate}
\end{lemma}

Now we have all the ingredients to study tightness and convergence
properties of controlled processes $\{{\bar{Z}}^n\}$ that are driven by
controls $\{\varphi^n=(\varphi_1^n,\varphi_2^n,\varphi_3^n)\}$ that satisfy %
\eqref{eq:cost_bd_repn}.

\begin{lemma}
\label{lem:tightness} Suppose that for some $M_0 <\infty$, $\{\varphi^n\}
\subset \tilde{\mathcal{A}}_{b,M_0}$.

\begin{enumerate}[(a)]

\item The sequence of random variables $\{(\varphi^n,{\bar{Z}}^n)\}$ is a
tight collection of ${\mathcal{M}} \times {\mathcal{E}}$ valued random
variables.

\item Suppose $(\varphi^n,{\bar{Z}}^n)$ converges along a subsequence, in
distribution, to $(\varphi,{\bar{Z}})$ given on some probability space $%
(\Omega^*,{\mathcal{F}}^*,P^*)$, where ${\bar{Z}} = ({\bar{X}},{\bar{Y}})$.
Then a.s.\ $P^*$, ${\bar{Z}} \in {\mathcal{C}}([0,T]:{\mathbb{R}}_+^2)$ and 
\begin{align}
\bar X(t) &= \Gamma \left(x + \lambda\int_0^{\cdot}\varphi_1(s)\, ds - \mu
\int_0^{\cdot}\varphi_2(s)\, ds - \theta \int_{[0,\cdot]\times {\mathbb{R}}%
_+} {\boldsymbol{1}}_{[0, \bar X(s)]}(z) \varphi_3(s,y) \,ds \,dy \right)
(t),  \label{eq:Xbar_sy} \\
\bar Y(t) &= \theta \int_{[0,t]\times {\mathbb{R}}_+} {\boldsymbol{1}}_{[0,
\bar X(s)]}(z) \varphi_3(s,y)\, ds\, dy, \quad t \in [0,T]
\label{eq:Ybar_sy}
\end{align}
\end{enumerate}
\end{lemma}

\begin{proof}
	By assumption $\{\varphi^n\}$ is a sequence of $\cls_{M_0}$ valued random variables and is therefore automatically tight. 
	Next, we argue $\Cmc$-tightness of $\{\bar Z^n\}$.
	Let $\Xtil^n(t)$ be the process appearing in the argument of $\Gamma$ in \eqref{eq:Xbar_n}, namely
	\begin{equation}
		\label{eq:Xtil_n}
		\Xtil^n(t) \doteq x_n + \frac{1}{n}  N_1^{n\varphi_1^n}(t)
			- \frac{1}{n} N_2^{n\varphi_2^n}(t) 
			- \frac{1}{n} \int_{[0,t]\times \R_+}  \one_{\left[0, \left(\Xbar^n(s-) - \frac{1}{n}\right)^+\right]}(y) \, N_3^{n\varphi_3^n}(ds\,dy).
	\end{equation}
	We will argue tightness of $\{\Xtil^n\}$ and get tightness of $\{\Xbar^n\}$ by continuity of the Skorohod map $\Gamma$.
	Write
	\begin{align*}
		\frac{1}{n}  N_1^{n\varphi_1^n}(t)  = \la\int_0^t \varphi_1^n(s)\,ds + \frac{1}{n}  \Ntil_1^{n\varphi_1^n}(t),\;\; &
		\frac{1}{n}  N_2^{n\varphi_2^n}(t)  = \mu\int_0^t \varphi_2^n(s)\,ds + \frac{1}{n}  \Ntil_2^{n\varphi_2^n}(t), \\
		\frac{1}{n} \int_{[0,t]\times \R_+}  \one_{\left[0, \left(\Xbar^n(s-) - \frac{1}{n}\right)^+\right]}(y) \, N_3^{n\varphi_3^n}(ds\,dy) & = \theta \int_{[0,t]\times \R_+} \one_{\left[0, \left(\Xbar^n(s) - \frac{1}{n}\right)^+\right]}(y) \varphi_3^n(s,y)\,ds\,dy \\
		& \quad + \frac{1}{n} \int_{[0,t]\times \R_+}  \one_{\left[0, \left(\Xbar^n(s-) - \frac{1}{n}\right)^+\right]}(y) \, \Ntil_3^{n\varphi_3^n}(ds\,dy).
	\end{align*}
	From \eqref{eq:Xbar_n} we have
	\begin{align*}
		E |\Xbar^n|_{*,t}^2 & \le \kappa_1 E |\Xtil^n|_{*,t}^2 \\
		& \le \kappa_2 + \kappa_2 E \left[ \int_0^t \left( \varphi_1^n(s) + \varphi_2^n(s) + \int_{\R_+} \one_{[0, \bar X^n(s)]}(y) \varphi_3^n(s,y)\,dy \right) ds \right]^2 \\
		& \qquad + \frac{\kappa_2}{n} E \int_0^t \left( \varphi_1^n(s) + \varphi_2^n(s) + \int_{\Rmb_+} \one_{[0, \Xbar^n(s)]}(y) \varphi_3^n(s,y) \,dy \right) ds \\
		& \le \kappa_2 + \kappa_2 E \left[ \int_0^t \left( \ell(\varphi_1^n(s))+2 + \ell(\varphi_2^n(s))+2 + \int_{\R_+} \one_{[0, \bar X^n(s)]}(y) (\ell(\varphi_3^n(s,y))+2)\,dy \right) ds \right]^2 \\
		& \qquad + \frac{\kappa_2}{n} E \int_0^t \left( \ell(\varphi_1^n(s))+2 + \ell(\varphi_2^n(s))+2 + \int_{\Rmb_+} \one_{[0, \Xbar^n(s)]}(y) (\ell(\varphi_3^n(s,y))+2) \,dy \right) ds \\
		& \le \kappa_3 + \kappa_3 \int_0^t E |\Xbar^n|_{*,s}^2 \,ds,
	\end{align*}
	where the first inequality uses the explicit expression for the Skorohod map $\Gamma$ in (\ref{eq:SM}), the second uses Doob's inequality, the third uses Lemma \ref{lem:property_ell}(b), and the last uses the fact that \eqref{eq:cost_bd_repn} is satisfied.
	It then follows from Gronwall's inequality that
	\begin{equation}
		\label{eq:Xbar_n_moment_bd}
		\sup_{n \in \Nmb} E |\Xbar^n|_{*,T}^2 < \infty, \quad \sup_{n \in \Nmb} E |\Xtil^n|_{*,T}^2 < \infty.
	\end{equation}
	This gives tightness of $\{\Xtil^n(t)\}$ for each fixed $t \in [0,T]$.
	As for the fluctuation of $\Xtil^n$, let $\mathcal{T}^\delta$ be the collection of all $[0,T-\delta]$ valued stopping times $\tau$ for each $\delta \in [0,T]$.
	In order to argue tightness of $\{\Xtil^n\}$,
	by the Aldous--Kurtz tightness criterion (cf.\ \cite[Theorem 2.7]{Kurtz1981approximation}) it suffices to show that
	\begin{equation}
		\label{eq:tightness_fluctuation}
		\limsup_{\delta \to 0} \limsup_{n \to \infty} \sup_{\tau \in \mathcal{T}^\delta} E|\Xtil^n(\tau+\delta) - \Xtil^n(\tau)| = 0.
	\end{equation}
	From \eqref{eq:Xtil_n} it follows that for every $M \in (0,\infty)$,
	\begin{align*}
		& E|\Xtil^n(\tau+\delta) - \Xtil^n(\tau)| \\
		& \le \kappa_4 E \int_\tau^{\tau+\delta} \left( \varphi_1^n(s) + \varphi_2^n(s) + \int_{\R_+} \one_{[0, \Xbar^n(s)]}(y) \varphi_3^n(s,y) \, dy \right) ds \\
		& \le \kappa_4 E \int_\tau^{\tau+\delta} \left( \varphi_1^n(s) \one_{\{\varphi_1^n(s)>M\}} + \varphi_2^n(s) \one_{\{\varphi_2^n(s)>M\}} + \int_{\R_+} \one_{[0, \Xbar^n(s)]}(y) \varphi_3^n(s,y) \one_{\{\varphi_3^n(s)>M\}} \, dy \right) ds \\
		& \qquad + \kappa_4 E \int_\tau^{\tau+\delta} \left( M + M + \int_{\R_+} \one_{[0, \Xbar^n(s)]}(y) M \, dy \right) ds \\
		& \le \kappa_5 \gamma(M) + \kappa_5 M \delta
	\end{align*}
	where the third inequality follows from Lemma \ref{lem:property_ell}(a), \eqref{eq:cost_bd_repn} and \eqref{eq:Xbar_n_moment_bd}.
	Therefore
	\begin{equation*}
		\limsup_{\delta \to 0} \limsup_{n \to \infty} \sup_{\tau \in \mathcal{T}^\delta} E|\Xtil^n(\tau+\delta) - \Xtil^n(\tau)| \le \kappa_5 \gamma(M).
	\end{equation*}
	Taking $M \to \infty$ gives \eqref{eq:tightness_fluctuation}.
	Combining this with \eqref{eq:Xbar_n_moment_bd} and \eqref{eq:Xbar_n} we have tightness of $\{(\Xtil^n,\bar X^n)\}$. 
	Similar estimate gives tightness of $\{\bar Y^n\}$.
	The $\Cmc$-tightness of $\{(\Xtil^n,\Zbar^n)\}$ is clear as its jump size is  $O(\frac{1}{n})$.
	
	Suppose now that $(\varphi^n, \Xtil^n,\Zbar^n)$ converge in distribution, along a subsequence, in  $\Mmc \times \Dmc([0,T]:\Rmb) \times \Emc$,
	to $(\varphi , \Xtil, \Zbar)$ given on some probability space $(\Omega^*,\Fmc^*,P^*)$, where $\Zbar=(\Xbar,\Ybar)$.
	Assume without loss of generality that the convergence holds along the whole sequence.
	From the $\Cmc$-tightness of $\{(\Xtil^n,\Zbar^n)\}$ we  have $(\Xtil,\Zbar) \in \Cmc([0,T]:\Rmb) \times \Cmc([0,T]:\Rmb_+^2)$ a.s.\ $P^*$.
	Using Doob's inequality, Lemma \ref{lem:property_ell}(b), \eqref{eq:cost_bd_repn} and \eqref{eq:Xbar_n_moment_bd} we have	
	\begin{align*}
		& E\left|\left(\frac{1}{n} \Ntil_1^{n\varphi_1^n}(\cdot)\right)^2 + \left(\frac{1}{n} \Ntil_2^{n\varphi_2^n}(\cdot)\right)^2 + \left(\frac{1}{n} \int_{[0,\cdot]\times \R_+}  \one_{\left[0, \left(\Xbar^n(s-) - \frac{1}{n}\right)^+\right]}(y) \, \Ntil_3^{n\varphi_3^n}(ds\,dy)\right)^2 \right|_{*,T} \\
		& \le \frac{\kappa_6}{n} E \int_0^T \left( \varphi_1^n(s) + \varphi_2^n(s) + \int_{\Rmb_+} \one_{[0, \Xbar^n(s)]}(y) \varphi_3^n(s,y) \,dy \right) ds \to 0
	\end{align*}
	as $n\to \infty$, where for $f:[0,T]\to \mathbb{R}$, $|f|_{*,T}\doteq \sup_{0\le s \le T}|f(s)|$.
	By appealing to the Skorohod representation theorem, we can further assume without loss of generality that the above convergence, $(\varphi^n, \Xtil^n,\Zbar^n) \to (\varphi , \Xtil, \Zbar)$ and $(\Xtil,\Zbar) \in \Cmc([0,T]:\Rmb) \times \Cmc([0,T]:\Rmb_+^2)$ hold a.s.\ $P^*$, namely on some set $G \in \Fmc^*$ such that $P^*(G^c)=0$.
	Fix $\omega^* \in G$.
	The rest of the argument will be made for such an $\omega^*$ which will be suppressed from the notation.

	In order to prove \eqref{eq:Xbar_sy} and \eqref{eq:Ybar_sy}, it suffices to show
	\begin{align}
		\int_0^t \varphi_1^n(s)\, ds + \int_0^t \varphi_2^n(s)\, ds & \to \int_0^t \varphi_1(s)\, ds + \int_0^t \varphi_2(s)\, ds, \label{eq:char_pf1} \\
		\int_{[0,t]\times \R_+} \one_{\left[0, \left(\Xbar^n(s-) - \frac{1}{n}\right)^+\right]}(y) \varphi_3^n(s,y)\,ds\,dy & \to \int_{[0,t]\times \R_+} \one_{[0, \bar X(s)]}(y) \varphi_3(s,y)\,ds\,dy, \label{eq:char_pf2}
	\end{align}
	as $n \to \infty$ for each $t \in [0,T]$.
	The convergence in \eqref{eq:char_pf1} is immediate from that of $\varphi^n \to \varphi \in \Mmc$.
	Next note that 
	\begin{align*}
		& \left| \int_{[0,t]\times \R_+} \one_{\left[0, \left(\Xbar^n(s-) - \frac{1}{n}\right)^+\right]}(y) \varphi_3^n(s,y)\,ds\,dy - \int_{[0,t]\times \R_+} \one_{[0, \bar X(s)]}(y) \varphi_3(s,y)\,ds\,dy \right| \\
		& \le \int_{[0,t]\times \R_+} \left| \one_{\left[0, \left(\Xbar^n(s-) - \frac{1}{n}\right)^+\right]}(y) - \one_{[0, \bar X(s)]}(y) \right| \varphi_3^n(s,y)\,ds\,dy \\
		& \quad + \left| \int_{[0,t]\times \R_+} \one_{[0, \bar X(s)]}(y) \left( \varphi_3^n(s,y) -  \varphi_3(s,y) \right) ds\,dy \right|.
	\end{align*}
	For $M \in (0,\infty)$, it follows from Lemma \ref{lem:property_ell}(a) and \eqref{eq:cost_bd_repn} that
	\begin{align*}
		& \int_{[0,t]\times \R_+} \left| \one_{\left[0, \left(\Xbar^n(s-) - \frac{1}{n}\right)^+\right]}(y) - \one_{[0, \bar X(s)]}(y) \right| \varphi_3^n(s,y)\,ds\,dy \\
		& \le \int_{[0,t]\times \R_+} \varphi_3^n(s,y) \one_{\{\varphi_3^n(s,y) > M\}} \,ds\,dy + \int_{[0,t]\times \R_+} \left| 
		\one_{\left[0, \left(\Xbar^n(s-) - \frac{1}{n}\right)^+\right]}(y) - \one_{[0, \bar X(s)]}(y) \right| M \,ds\,dy \\
		& \le \gamma(M) M_0 + MT |\bar X^n - \bar X|_{*,T} + MT/n,
	\end{align*}
	which converges to zero on sending $n\to \infty$ and then $M \to \infty$.
	Since the Lebesgue measure of $\{(s,y) \in [0,T]\times\Rmb_+ : y=\Xbar(s)\}$ is zero, it follows from $\varphi^n \to \varphi \in \Mmc$ that
	$$\left| \int_{[0,t]\times \R_+} \one_{[0, \bar X(s)]}(y) \left( \varphi_3^n(s,y) -  \varphi_3(s,y) \right) ds\,dy \right| \to 0$$
	as $n \to \infty$.
	Combining above three displays gives \eqref{eq:char_pf2}.
	This completes the proof.
\end{proof}

\subsection{Proof of Laplace upper bound.}
\label{sec:lapuppaband}

In this section we prove \eqref{eq:lapuppaband} for a fixed $h \in C_b(%
\mathcal{E})$. We begin by giving an alternative representation of $I_T$
that will be convenient for the proof of this inequality. For $(\xi, \zeta)
\in \mathcal{C}([0,T]: {\mathbb{R}}_+^2)$, let $\tilde{\mathcal{U}}(\xi,
\zeta) $ be the collection of all measurable $\tilde \varphi = (\tilde
\varphi_1, \tilde \varphi_2, \tilde \varphi_3)$ such that $\tilde \varphi_i
\colon [0,T] \to {\mathbb{R}}_+$, $i=1,2$, $\tilde \varphi_3 \colon
[0,T]\times {\mathbb{R}}_+ \to {\mathbb{R}}_+$ and the following equations
are satisfied for $t \in [0,T]$: 
\begin{align*}
\xi(t) &= \Gamma\left (x_0 + \lambda\int_0^{\cdot} \tilde \varphi_1(s) \,ds
- \mu \int_0^{\cdot} \tilde \varphi_2(s)\, ds - \theta \int_{[0,\cdot]\times 
{\mathbb{R}}_+} {\tilde{\varphi}}_3(s, y) {\boldsymbol{1}}_{[0,\xi(s)]}(y)
\,ds\, dy\right)(t), \\
\zeta(t) &= \theta \int_{[0,t]\times {\mathbb{R}}_+} {\tilde{\varphi}}_3(s,
y) {\boldsymbol{1}}_{[0,\xi(s)]}(y)\, ds\, dy .
\end{align*}
Define for $(\xi, \zeta) \in \mathcal{C}([0,T]: {\mathbb{R}}_+^2)$ 
\begin{equation}\label{eq:ratefn2}
\tilde I_T(\xi, \zeta) = \inf_{\tilde \varphi \in \tilde{\mathcal{U}}(\xi,
\zeta)}\left\{ \lambda\int_0^T \ell(\tilde \varphi_1(s)) \,ds + \mu\int_0^T
\ell(\tilde \varphi_2(s)) \,ds + \theta\int_{[0,T]\times {\mathbb{R}}_+}
\ell(\tilde \varphi_3(s,y)) \,ds \,dy\right\}.
\end{equation}
We set $\tilde I_T(\xi, \zeta) \doteq \infty$ if $\tilde{\mathcal{U}}(\xi,
\zeta)$ is empty or $(\xi, \zeta) \in \mathcal{D}([0,T]: {\mathbb{R}}%
_+^2)\setminus \mathcal{C}([0,T]: {\mathbb{R}}_+^2)$. 
\begin{remark}
	\label{rem:abscty}
	{Note that if for $(\xi, \zeta) \in \mathcal{C}([0,T]: {\mathbb{R}}_+^2)$, $\tilde I_T(\xi, \zeta) < \infty$ and
	$\tilde \varphi \in \tilde{\mathcal{U}}(\xi,
\zeta)$ is such that its cost given by the right side in \eqref{eq:ratefn2} is finite, then from the superlinearity of $\ell$ we must have that
$$\int_0^T \tilde \varphi_i(s) ds < \infty, \; i= 1,2, \mbox{ and } \int_{[0,T]\times [0,M]}
\tilde \varphi_3(s,y) \,ds \,dy < \infty \mbox{ for all } M <\infty.$$
This says that 
$\zeta$ is absolutely continuous and $\xi = \Gamma(\check \xi)$ for some $\check \xi$ which is absolutely continuous as well.
It follows that $t \mapsto \inf_{0\le s \le t} \check \xi(s)\wedge 0$ is also absolutely continuous. Thus we have that
$\zeta, \xi$ and $\xi - \check \xi$ are all absolutely continuous. This fact will be used in the proof of Proposition \ref{prop:uniq}.
}
\end{remark}

Then we have the
following result.

\begin{lemma}
\label{lem:iitil} For all $T\ge 0$ and $(\xi, \zeta) \in \mathcal{D}([0,T]: {%
\mathbb{R}}_+^2)$, $I_T(\xi, \zeta) = \tilde I_T(\xi, \zeta)$.
\end{lemma}

\begin{proof}
	Consider $T\ge 0$ and for $(\xi, \zeta) \in \calC([0,T]: \R_+^2)$, $\tilde \varphi = (\tilde \varphi_1, \tilde \varphi_2, \tilde \varphi_3)\in \tilde \calU(\xi, \zeta)$.
	Define
	$\varphi_3:[0,T] \to \R_+$ by
	$$
	\varphi_3(s) \doteq \one_{\{\xi(s)\neq 0\}} \frac{1}{\xi(s)} \int_{[0,\xi(s)]}\tilde \varphi_3(s,y)\, dy + \one_{\{\xi(s) = 0\}}, \quad s \in [0,T].$$
	Note that
	$$\int_{[0,t]\times \R_+} \tilde \varphi_3(s,y) \one_{[0,\xi(s)]}(y) \,ds \,dy
	= \int_{[0,t]}  \xi(s)\varphi_3(s)\, ds \mbox{ for all } t \in [0,T]$$
	and, by convexity of $\ell$,
	{\begin{align*}
		\int_{[0,T]\times \R_+} \ell(\tilde \varphi_3(s,y)) \,ds \,dy & \ge \int_{[0,T]\times \R_+} \one_{\{\xi(s)\neq 0\}} \one_{[0,\xi(s)]}(y) \ell(\tilde \varphi_3(s,y)) \,ds \,dy \\
		& = \int_{[0,T]} \one_{\{\xi(s)\neq 0\}} \xi(s) \left( \frac{1}{\xi(s)} \int_{[0,\xi(s)]} \ell(\tilde \varphi_3(s,y)) \,dy \right) ds  \\
		& \ge \int_{[0,T]} \one_{\{\xi(s)\neq 0\}} \xi(s)\ell \left( \frac{1}{\xi(s)} \int_{[0,\xi(s)]} \tilde \varphi_3(s,y) \,dy \right) ds \\
		& = \int_{[0,T]} \xi(s)\ell(\varphi_3(s))\,ds.
	\end{align*}}
	Clearly $\varphi \doteq  (\tilde \varphi_1, \tilde \varphi_2,  \varphi_3) \in \calU(\xi, \zeta)$.
	This shows $I_T(\xi, \zeta) \leq \tilde I_T(\xi, \zeta)$. 
	The reverse inequality is immediate on observing that if $( \varphi_1, \varphi_2,  \varphi_3) \in \calU(\xi, \zeta)$ then, with $\tilde \varphi_3(s,y)\doteq \varphi_3(s)$ for $y\in [0,1]$,
	$( \varphi_1,  \varphi_2,  \tilde\varphi_3) \in \tilde\calU(\xi, \zeta)$ and the costs for the two controls are identical.
	The result follows.
\end{proof}

We now return to the proof of \eqref{eq:lapuppaband}. For each $n$, let $%
\varphi ^{n}\doteq (\varphi _{1}^{n},\varphi _{2}^{n},\varphi _{3}^{n})$ be $%
\frac{1}{n}$-optimal in \eqref{eq:mainrepn}, namely, 
\begin{align}
& -\frac{1}{n}\log E\exp \left\{ -nh(Z^{n})\right\}  \notag \\
& \quad \geq E\left[ h(\bar{Z}^{n})+\int_{0}^{T}\left( \lambda \ell (\varphi
_{1}^{n}(s))+\mu \ell (\varphi _{2}^{n}(s))+\theta \int_{{\mathbb{R}}%
_{+}}\ell (\varphi _{3}^{n}(s,y))\,dy\right) ds\right] -\frac{1}{n},
\label{eq:vepsbd}
\end{align}%
where $\bar{Z}^{n}$ is as introduced below \eqref{eq:mainrepn}. Using the
boundedness of $h$, by a standard
localization
argument (see \cite[Proof of Theorem 4.2]{BudhirajaDupuisMaroulas2011variational}), we can assume without loss of generality that %
\eqref{eq:cost_bd_repn} is satisfied for some $M_{0}<\infty $.

We have from Lemma \ref{lem:tightness}(a) that $\{(\varphi ^{n},{\bar{Z}}%
^{n})\}$ is tight in ${\mathcal{M}}\times {\mathcal{E}}$. Assume without
loss of generality that $(\varphi ^{n},{\bar{Z}}^{n})$ converges along the
whole sequence, in distribution, to $(\varphi ,{\bar{Z}})\in {\mathcal{M}}%
\times {\mathcal{E}}$ given on some probability space $(\Omega ^{\ast },{%
\mathcal{F}}^{\ast },P^{\ast })$, where ${\bar{Z}}=({\bar{X}},{\bar{Y}})$.
By Lemma \ref{lem:tightness}(b) we have ${\bar{Z}}\in {\mathcal{C}}([0,T]:{%
\mathbb{R}}_{+}^{2})$ and $\varphi \in \tilde{\mathcal{U}}({\bar{X}},{\bar{Y}%
})$ a.s.\ $P^{\ast }$. From \eqref{eq:vepsbd}, Fatou's lemma and the lower semicontinuity established in 
 \cite[Lemma A.1]{BudhirajaChenDupuis2013large} (see proof of (A.1) therein)
it follows that 
\begin{align}
& \liminf_{n\rightarrow \infty }-\frac{1}{n}\log E\exp \left\{
-nh(Z^{n})\right\}  \notag \\
& \quad \geq E\left[ h(\bar{Z})+\int_{0}^{T}\left( \lambda \ell (\varphi
_{1}(s))+\mu \ell (\varphi _{2}(s))+\theta \int_{{\mathbb{R}}_{+}}\ell
(\varphi _{3}(s,y))\,dy\right) ds\right] .  \label{eq:upper_pf1}
\end{align}

Thus we have that 
\begin{align*}
\liminf_{n \to \infty} -\frac{1}{n}\log E\exp \left\{ -nh(Z^{n})\right\} & \geq \inf_{\phi
\in \mathcal{E}}E\left[ h(\phi )+\tilde{I}_{T}(\phi )\right] \\
& =\inf_{\phi \in \mathcal{E}}E\left[ h(\phi )+I_{T}(\phi )\right] ,
\end{align*}%
where the last equality is from Lemma \ref{lem:iitil}. The proof of the
upper bound is complete. \hfill \qed

\subsection{Proof of Laplace lower bound.}
\label{sec:laplowbd}

In this section we prove the inequality \eqref{eq:laplowaband} for a fixed $%
h \in C_b({\mathcal{E}})$.

We begin by establishing a key uniqueness property.

\begin{proposition}
\label{prop:uniq} Fix $(\xi, \zeta) \in \mathcal{C}([0,T]: {\mathbb{R}}_+^2) 
$ with ${\tilde{I}}_T(\xi, \zeta)<\infty$ and $\varphi \in \tilde{\mathcal{U}}%
(\xi, \zeta)$ be such that
{the cost associated with $\varphi$ given by the right side of \eqref{eq:ratefn2} is finite}. If $(\tilde \xi, \tilde \zeta) \in \mathcal{C}([0,T]: {%
\mathbb{R}}_+^2)$ is another pair such that $\varphi \in \tilde{\mathcal{U}}%
(\tilde \xi, \tilde \zeta)$,  then $(\xi, \zeta) = (\tilde \xi, \tilde \zeta) 
$.
\end{proposition}

We first complete the proof of the lower bound assuming that the Proposition %
\ref{prop:uniq} holds. Fix $\sigma \in (0,1)$ and choose $(\xi ^{\ast
},\zeta ^{\ast })\in \mathcal{C}([0,T]:{\mathbb{R}}_{+}^{2})$ such that 
\begin{equation*}
h(\xi ^{\ast },\zeta ^{\ast })+I_{T}(\xi ^{\ast },\zeta ^{\ast })\leq
\inf_{(\xi ,\zeta )\in {\mathcal{E}}}\{h(\xi ,\zeta )+I_{T}(\xi ,\zeta
)\}+\sigma .
\end{equation*}%
Since ${\tilde{I}}_{T}(\xi ^{\ast },\zeta ^{\ast })=I_{T}(\xi ^{\ast },\zeta
^{\ast })<\infty $, we can choose $\varphi ^{\ast }\in \tilde{\mathcal{U}}%
(\xi ^{\ast },\zeta ^{\ast })$ such that 
\begin{equation*}
\int_{0}^{T}\left( \lambda \ell (\varphi _{1}^{\ast }(s))+\mu \ell (\varphi
_{2}^{\ast }(s))+\theta \int_{{\mathbb{R}}_{+}}\ell (\varphi _{3}^{\ast
}(s,y))\,dy\right) ds\leq I_{T}(\xi ^{\ast },\zeta ^{\ast })+\sigma .
\end{equation*}%
Define the deterministic controls 
\begin{equation*}
\varphi _{i}^{n}(s)\doteq \frac{1}{n}{\boldsymbol{1}}_{\{\varphi _{i}^{\ast
}(s)\leq 1/n\}}+\varphi _{i}^{\ast }(s){\boldsymbol{1}}_{\{1/n<\varphi
_{i}^{\ast }(s)<n\}}+n{\boldsymbol{1}}_{\{\varphi _{i}^{\ast }(s)\geq n\}}
\end{equation*}%
for $i=1,2$ and $s\in \lbrack 0,T]$, and 
\begin{equation*}
\varphi _{3}^{n}(s,y)\doteq \frac{1}{n}{\boldsymbol{1}}_{\{\varphi
_{3}^{\ast }(s,y)\leq 1/n,y\leq n\}}+\varphi _{3}^{\ast }(s,y){\boldsymbol{1}%
}_{\{1/n<\varphi _{3}^{\ast }(s,y)<n,y\leq n\}}+n{\boldsymbol{1}}_{\{\varphi
_{3}^{\ast }(s,y)\geq n,y\leq n\}}+{\boldsymbol{1}}_{\{y>n\}}
\end{equation*}%
for $(s,y)\in \lbrack 0,T]\times {\mathbb{R}}_{+}$. Then $\varphi ^{n}\doteq
(\varphi _{1}^{n},\varphi _{2}^{n},\varphi _{3}^{n})\in {\bar{{\mathcal{A}}}}%
_{b}$. We will use that $\ell (z)\geq 0$ and $\ell (1)=0$, and also that $%
\ell (z)$ is increasing for $z>1$ and decreasing for $z<1$. Then from %
\eqref{eq:mainrepn} and these properties we have 
\begin{align}
-\frac{1}{n}\log E\exp \left\{ -nh(Z^{n})\right\} & \leq E\left[ h(\bar{Z}%
^{n})+\int_{0}^{T}\left( \lambda \ell (\varphi _{1}^{n}(s))+\mu \ell
(\varphi _{2}^{n}(s))+\theta \int_{{\mathbb{R}}_{+}}\ell (\varphi
_{3}^{n}(s,y))\,dy\right) ds\right]   \notag \\
& \leq E\left[ h(\bar{Z}^{n})+\int_{0}^{T}\left( \lambda \ell (\varphi
_{1}^{\ast }(s))+\mu \ell (\varphi _{2}^{\ast }(s))+\theta \int_{{\mathbb{R}}%
_{+}}\ell (\varphi _{3}^{\ast }(s,y))\,dy\right) ds\right] .
\label{eq:lowbd1}
\end{align}%
Note that \eqref{eq:cost_bd_repn} is satisfied with $M_{0}$ replaced by $%
M\doteq (I_{T}(\xi ^{\ast },\zeta ^{\ast })+1)/\min (\lambda ,\mu ,\theta )$%
. It then follows from Lemma \ref{lem:tightness}(a) that $\{(\varphi ^{n},%
\bar{Z}^{n})\}$ is tight. Clearly $\varphi ^{n}\rightarrow \varphi ^{\ast
}\in {\mathcal{M}}$. By Lemma \ref{lem:tightness}(b), if $\bar{Z}^{n}$
converges along a subsequence to $\bar{Z}$, then $\bar{Z}$ must satisfy %
\eqref{eq:Xbar_sy} and \eqref{eq:Ybar_sy} with $\varphi $ replaced with $%
\varphi ^{\ast }$, namely $\varphi ^{\ast }\in \tilde{{\mathcal{U}}}(\bar{Z})
$. From uniqueness in Proposition \ref{prop:uniq} it follows that $\bar{Z}%
=(\xi ^{\ast },\zeta ^{\ast })$. Finally, from \eqref{eq:lowbd1} we have 
\begin{align*}
& \limsup_{n\rightarrow \infty }-\frac{1}{n}\log E\exp \left\{
-nh(Z^{n})\right\}  \\
& \quad \leq \limsup_{n\rightarrow \infty }E\left[ h(\bar{Z}%
^{n})+\int_{0}^{T}\left( \lambda \ell (\varphi _{1}^{\ast }(s))+\mu \ell
(\varphi _{2}^{\ast }(s))+\theta \int_{{\mathbb{R}}_{+}}\ell (\varphi
_{3}^{\ast }(s,y))\,dy\right) ds\right]  \\
& \quad =h(\xi ^{\ast },\zeta ^{\ast })+\int_{0}^{T}\left( \lambda \ell
(\varphi _{1}^{\ast }(s))+\mu \ell (\varphi _{2}^{\ast }(s))+\theta \int_{{%
\mathbb{R}}_{+}}\ell (\varphi _{3}^{\ast }(s,y))\,dy\right) ds \\
& \quad \leq h(\xi ^{\ast },\zeta ^{\ast })+I_{T}(\xi ^{\ast },\zeta ^{\ast
})+\sigma  \\
& \quad \leq \inf_{(\xi ,\zeta )\in {\mathcal{E}}}\{h(\xi ,\zeta )+I_{T}(\xi
,\zeta )\}+2\sigma ,
\end{align*}%
which completes the proof of the lower bound. \hfill \qed

\medskip
We now return to the proof of Proposition \ref{prop:uniq}.

\textbf{Proof of Proposition \ref{prop:uniq}.} Suppose that $(%
\tilde{\xi},\tilde{\zeta})$ is another pair such that $\varphi \in \tilde{%
\mathcal{U}}(\tilde{\xi},\tilde{\zeta})$. It suffices to show that $\tilde{%
\xi}=\xi $. Write $\xi (t)=\psi (t)+\eta (t)$, $\tilde{\xi}(t)=\tilde{\psi}%
(t)+\tilde{\eta}(t)$, where $\psi ,{\tilde{\psi}}$ are the unconstrained
processes defined by 
\begin{align*}
\psi (t)& =x_{0}+\lambda \int_{0}^{t}\varphi _{1}(s)\,ds-\mu
\int_{0}^{t}\varphi _{2}(s)\,ds-\theta \int_{\lbrack 0,t]\times {\mathbb{R}}%
_{+}}\varphi _{3}(s,y){\boldsymbol{1}}_{[0,\xi (s)]}(y)\,ds\,dy, \\
{\tilde{\psi}}(t)& =x_{0}+\lambda \int_{0}^{t}\varphi _{1}(s)\,ds-\mu
\int_{0}^{t}\varphi _{2}(s)\,ds-\theta \int_{\lbrack 0,t]\times {\mathbb{R}}%
_{+}}\varphi _{3}(s,y){\boldsymbol{1}}_{[0,{\tilde{\xi}}(s)]}(y)\,ds\,dy,
\end{align*}%
and $\eta ,\tilde{\eta}$ are the reflection terms such that (see e.g., \cite[%
Section 3.6.C]{KaratzasShreve1991brownian}) 
\begin{align*}
& \eta (0)=0,\eta (t)\mbox{ is nondecreasing and }\int_{0}^{T}\xi
(t)\,d\eta (t)=0, \\
& {\tilde{\eta}}(0)=0,{\tilde{\eta}}(t)\mbox{ is nondecreasing and }%
\int_{0}^{T}{\tilde{\xi}}(t)\,d{\tilde{\eta}}(t)=0.
\end{align*}%
{Recall from Remark \ref{rem:abscty} that $\xi, \tilde \xi, \psi, \tilde \psi, \eta, \tilde \eta$ are absolutely continuous.
For an absolutely continuous function $f$, we will denote its a.e.\ derivative as $f'$.} Then for $t\in \lbrack 0,T]$, 
\begin{align*}
\lbrack \xi (t)-\tilde{\xi}(t)]^{2}& =2\int_{0}^{t}(\xi (s)-\tilde{\xi}%
(s))(\xi ^{\prime }(s)-\tilde{\xi}^{\prime }(s))\,ds \\
& =2\int_{0}^{t}(\xi (s)-\tilde{\xi}(s))(\psi ^{\prime }(s)-\tilde{\psi}%
^{\prime }(s))\,ds+2\int_{0}^{t}(\xi (s)-\tilde{\xi}(s))d(\eta (s)-\tilde{%
\eta}(s)).
\end{align*}%
Note that for a.e.\ $s\in \lbrack 0,T]$, whenever $\xi (s)>{\tilde{\xi}}(s)$%
, $\xi (s)={\tilde{\xi}}(s)$, or $\xi (s)<{\tilde{\xi}}(s)$, 
\begin{equation*}
(\xi (s)-\tilde{\xi}(s))(\psi ^{\prime }(s)-\tilde{\psi}^{\prime
}(s))=-\theta (\xi (s)-\tilde{\xi}(s))\int_{{\mathbb{R}}_{+}}\varphi
_{3}(s,y)[{\boldsymbol{1}}_{[0,\xi (s)]}(y)-{\boldsymbol{1}}_{[0,\tilde{\xi}%
(s)]}(y)]\,dy\leq 0.
\end{equation*}%
Also, since $\tilde{\eta}$ is nondecreasing, 
\begin{align*}
\int_{0}^{t}{\boldsymbol{1}}_{\{\xi (s)>\tilde{\xi}(s)\}}(\xi (s)-\tilde{\xi}%
(s))d(\eta (s)-\tilde{\eta}(s))& \leq \int_{0}^{t}{\boldsymbol{1}}_{\{\xi
(s)>\tilde{\xi}(s)\}}(\xi (s)-\tilde{\xi}(s))d\eta (s) \\
& \leq \int_{0}^{t}{\boldsymbol{1}}_{\{\xi (s)>0\}}\xi (s)d\eta (s) \\
& =0.
\end{align*}%
Similarly 
\begin{equation*}
\int_{0}^{t}{\boldsymbol{1}}_{\{\xi (s)<\tilde{\xi}(s)\}}(\xi (s)-\tilde{\xi}%
(s))d(\eta (s)-\tilde{\eta}(s))\leq 0.
\end{equation*}%
Thus $[\xi (t)-\tilde{\xi}(t)]^{2}=0$ for all $t\in \lbrack 0,T]$. This
completes the proof. \hfill \qed

\section{Properties of the rate function.}

\label{sec:proprf} For the rest of this paper we will assume that $\lambda
\ge \mu$. This property will not be explicitly noted in the statements of
various results.

In this section we give a different representation for the rate function $%
I_T $ and establish certain convexity properties of the rate function.

Let ${\mathcal{C}}_T$ be the collection of functions $(\xi,\zeta) \in {%
\mathcal{C}}([0,T]:{\mathbb{R}}_+^2)$ such that

\begin{enumerate}[(i)]

\item $\xi(0)=x_0$, $\zeta(0)=0$.

\item $\xi,\zeta$ are absolutely continuous, and $\zeta^{\prime }(t)=0$ for
a.e. $t \in [0,T]$ such that $\xi(t)=0$.

\item $\zeta$ is nondecreasing.
\end{enumerate}

We note that $I_T(\xi,\zeta)= \infty$ if $(\xi, \zeta) \not \in {\mathcal{C}}%
_T$.

For $(x,p,q)\in {\mathbb{R}}_{+}\times {\mathbb{R}}\times {\mathbb{R}}%
_{+}$, let
\begin{align*}
L(x,p,q)& \doteq \lambda \ell \left( \frac{\sqrt{(p+q)^{2}+4\lambda \mu }%
+(p+q)}{2\lambda }\right) +\mu \ell \left( \frac{\sqrt{(p+q)^{2}+4\lambda
\mu }-(p+q)}{2\mu }\right)  \\
& \quad +\theta x\ell \left( \frac{q}{\theta x}\right) \cdot {\boldsymbol{1}}%
_{\{x>0\}}+\infty \cdot {\boldsymbol{1}%
}_{\{x=0,q>0\}}.
\end{align*}%
The following lemma gives an alternative representation for $I_{T}(\xi
,\zeta )$ for $(\xi ,\zeta )\in {\mathcal{C}}_{T}$.

\begin{lemma}
\label{lem:rewrite} For $(\xi,\zeta) \in {\mathcal{C}}_T$, 
\begin{align}
I_T(\xi,\zeta) & = \int_0^T L(\xi(s),\xi^{\prime }(s),\zeta^{\prime
}(s)) \, ds.  \label{eq:itlrepn}
\end{align}
\end{lemma}

\begin{proof}
	Fix $(\xi,\zeta) \in \Cmc_T$.
	Recall 
	\begin{equation}
		I_T(\xi,\zeta) = \inf_{\varphi \in \Umc(\xi,\zeta)} \left\{ \lambda \int_0^T \ell(\varphi_1(s))\,ds + \mu \int_0^T \ell(\varphi_2(s))\,ds + \theta \int_0^T \xi(s) \ell(\varphi_3(s))\,ds \right\}.
		\label{eq:258}
	\end{equation}
	We would like to find a $\varphi \in \Umc(\xi,\zeta)$ for which the above infimum is achieved.
	
	Consider $t \in [0,T]$ such that $\xi'(t)$ and $\zeta'(t)$ exist.
	When $\xi(t) > 0$, from \eqref{eq:eqxizet} we have
	\begin{align}
		\xi'(t) + \zeta'(t) & = \lambda \varphi_1(t) - \mu \varphi_2(t), \label{eq:equation1}\\
		\zeta'(t) & = \theta \varphi_3(t) \xi(t). \label{eq:equation2}
	\end{align}
	It is easy to check that, given these constraints, the $\varphi(t)$   that minimizes 
	$$\lambda \ell(\varphi_1(t)) + \mu \ell(\varphi_2(t)) + \theta \xi(t) \ell(\varphi_3(t))$$
	must satisfy $\varphi_1(t)\varphi_2(t)=1$, namely
	\begin{align}
		\varphi_1(t) & = \frac{\sqrt{(\xi'(t) + \zeta'(t))^2 + 4\lambda\mu} + (\xi'(t) + \zeta'(t))}{2\lambda}, \label{eq:varphi1} \\
		\varphi_2(t) & = \frac{\sqrt{(\xi'(t) + \zeta'(t))^2 + 4\lambda\mu} - (\xi'(t) + \zeta'(t))}{2\mu}, \label{eq:varphi2} \\
		\varphi_3(t) & = \frac{\zeta'(t)}{\theta \xi(t)}. \label{eq:varphi3}
	\end{align}
	
	Since the Lebesgue measure of $\{t: \xi(t) =0, \xi'(t) \neq 0\}=0$,   we must have that for a.e. $t$ on the set $\{\xi(t)=0\}$, $\xi'(t) = 0$, $\zeta'(t)=0$, and, using \eqref{eq:eqxizet}, 
	\begin{equation*}
		\lambda \varphi_1(t) - \mu \varphi_2(t) \le 0.
	\end{equation*}
	Using Lagrange multipliers, it is easy to check that under this constraint, the minimizer of
	$$\lambda \ell(\varphi_1(t)) + \mu \ell(\varphi_2(t)) + \theta \xi(t) \ell(\varphi_3(t)) = \lambda \ell(\varphi_1(t)) + \mu \ell(\varphi_2(t))$$
	must satisfy $\varphi_1(t)\varphi_2(t)=1$ and $\lambda \varphi_1(t) - \mu \varphi_2(t) = 0$, namely
	\begin{equation}
		\label{eq:varphi0}
		\varphi_1(t) = \sqrt{\frac{\mu}{\lambda}}, \quad
		\varphi_2(t) = \sqrt{\frac{\lambda}{\mu}}, \quad
		\varphi_3(t) = 1.
	\end{equation}
	Since $\lambda \varphi_1(t) - \mu \varphi_2(t) = 0$, the $\varphi$ defined by \eqref{eq:varphi1} -- \eqref{eq:varphi0} is in $\Umc(\xi,\zeta)$
	and in fact \eqref{eq:eqxizet} holds without the Skorohod map on the right side of the first line.
	Plugging the minimizer $\varphi$ in the cost in \eqref{eq:258} gives the desired result.
\end{proof}

\begin{remark}
\label{rmk:rewrite}

\begin{enumerate}[(i)]
\item The proof of Lemma \ref{lem:rewrite} in fact shows that for $%
(\xi,\zeta) \in {\mathcal{C}}_T$ and a.e. $s \in [0,T]$ 
\begin{equation*}
L(\xi(s),\xi^{\prime }(s),\zeta^{\prime }(s)) = \lambda
\ell(\varphi_1(s)) + \mu \ell(\varphi_2(s)) + \theta \xi(s)
\ell(\varphi_3(s)),
\end{equation*}
where $\varphi = (\varphi_1, \varphi_2, \varphi_3)$ is given by %
\eqref{eq:varphi1} -- \eqref{eq:varphi0}.

\item The proof also shows that the representation \eqref{eq:itlrepn} holds
for any subinterval of $[0,T]$, namely 
\begin{align*}
& \inf_{\varphi \in {\mathcal{U}}(\xi,\zeta)} \left\{ \lambda
\int_{t_1}^{t_2} \ell(\varphi_1(s))\,ds + \mu \int_{t_1}^{t_2}
\ell(\varphi_2(s))\,ds + \theta \int_{t_1}^{t_2} \xi(s)
\ell(\varphi_3(s))\,ds \right\} \\
& \quad = \int_{t_1}^{t_2} L(\xi(s),\xi^{\prime }(s),\zeta^{\prime
}(s)) \, ds
\end{align*}
for each $(\xi,\zeta) \in {\mathcal{C}}_T$ and $0 \le t_1 < t_2 \le T$.
\end{enumerate}
\end{remark}

For general $(\xi,\zeta) \in {\mathcal{C}}_T$, let ${\mathcal{U}}%
_+(\xi,\zeta)$ be the collection of $\varphi = (\varphi_1, \varphi_2,
\varphi_3)$ such that $\varphi_i: [0,T] \to {\mathbb{R}}_+$, $i=1,2,3$ are
measurable maps and 
\begin{align*}
\xi(t) & = x_0 + \lambda \int_0^t \varphi_1(s)\,ds - \mu \int_0^t
\varphi_2(s)\,ds - \theta \int_0^t \varphi_3(s) \xi(s)\,ds, \\
\zeta(t) & = \theta \int_0^t \varphi_3(s) \xi(s)\,ds.
\end{align*}
As noted in the last line of the proof of Lemma \ref{lem:rewrite}, 
we have in fact proved the following result.

\begin{corollary}
\label{lem:no_Skorohod} For $(\xi,\zeta) \in {\mathcal{C}}_T$, 
\begin{equation*}
I_T(\xi,\zeta) = \inf_{\varphi \in {\mathcal{U}}_+(\xi,\zeta)} \left\{
\lambda \int_0^T \ell(\varphi_1(s))\,ds + \mu \int_0^T
\ell(\varphi_2(s))\,ds + \theta \int_0^T \xi(s) \ell(\varphi_3(s))\,ds
\right\}.
\end{equation*}
Moreover, the minimizer $\varphi$ in the above infimum is given by %
\eqref{eq:varphi1}--\eqref{eq:varphi3} when $\xi(t)>0$ and \eqref{eq:varphi0}
when $\xi(t)=0$.
\end{corollary}

The following lemma says that $L$ and $I_T(\xi,\zeta)$ are convex. 

\begin{lemma}
\label{lem:convexity} $L$ is a convex function on ${\mathbb{R}}_+ \times {%
\mathbb{R}} \times {\mathbb{R}}_+$ and $I_T$ is a convex function on ${%
\mathcal{C}}_T$.
\end{lemma}

\begin{proof}
	In view of Lemma \ref{lem:rewrite},
	it suffices to show convexity of $L$.
	Note that the function $$f_1(x) \doteq x \log \frac{\sqrt{x^2+4\lambda\mu}+x}{2\lambda} - \sqrt{x^2+4\lambda\mu}$$ is convex in $x \in \Rmb$ since $f_1''(x) = \frac{1}{\sqrt{x^2+4\lambda\mu}} > 0$.
	Also,
	\begin{align*}
		& \lambda \ell\left(\frac{\sqrt{(p + q)^2 + 4\lambda\mu} + (p + q)}{2\lambda}\right) + \mu \ell\left(\frac{\sqrt{(p + q)^2 + 4\lambda\mu} - (p + q)}{2\mu}\right) = f_1(p+q) + \lambda + \mu
	\end{align*}
	so the map taking $(p,q)$ to the left side of the last display is convex on $\R \times \R_+$. 
	
	Next, consider the function $f_2(x,y)\doteq x \ell(\frac{y}{x})$ defined on the convex set $O \doteq \{ (x,y) : x>0, y \ge 0 \}$.
	Note that the Hessian matrix of $f_2$ is
	\begin{equation*}
		\left( \begin{array}{cc}
			\frac{y}{x^2} & -\frac{1}{x} \\
			-\frac{1}{x} & \frac{1}{y} \\
		\end{array} \right)
	\end{equation*}
	on $\{ (x,y) : x>0, y>0 \}$, which is positive semidefinite.
	So $f_2$ is convex on $\{ (x,y) : x>0, y>0 \}$ and,
	since $f_2$ is continuous on $O$, it is also convex on $O$.	
	Let $\fbar_2$ be the extension of $f_2$ such that $\fbar_2(0,0)=0$ and $\fbar_2(x,y) = \infty$ for $(x,y) \in \R_+^2 \setminus \left(O \cup\{(0,0)\}\right)$.
	Note that for $(x,y)=(0,0)$, $(\xtil,\ytil) \in O$, and $r \in (0,1)$,
	\begin{align*}
		\fbar_2((1-r)x+r\xtil,(1-r)y+r\ytil) = r\xtil \ell({\ytil}/{\xtil}) = (1-r)\fbar_2(x,y) + r\fbar_2(\xtil,\ytil).
	\end{align*}
	Therefore $\fbar_2$ is convex on $\Rmb^2_+$ and so
	\begin{equation*}
		(x,q)\mapsto \theta x \ell\left(\frac{q}{\theta x}\right) \cdot \one_{\{x>0\}} + \infty \cdot \one_{\{x=0,q>0\}} = \fbar_2(\theta x,q)
	\end{equation*}
	is convex on $\Rmb_+^2$. We have thus shown the convexity of $L$ on $\Rmb_+ \times \Rmb \times \Rmb_+$ and  consequently that of $I_T$ on  $\Cmc_T$.
\end{proof}

\section{Construction of Minimizer for $I^*_{\protect\gamma,T}$.}

Recall that we are assuming that $\lambda \ge \mu$. In order to analyze the
long time asymptotics of $I_{\gamma, T}^*, I_{\gamma, T}$ and $\tilde
I_{\gamma, T}$, introduced in \eqref{eq:eqiii}, we will first formally
calculate a candidate minimizer $(\xi,\zeta) \in {\mathcal{C}}_T$ for $%
I_{\gamma, T}^*$, and then use that to prove our main result, Theorem \ref%
{thm:I*}.

\subsection{A Formal Calculation of a Candidate Minimizer.}

\label{sec:formal} Consider the Euler-Lagrange equations \cite[Chapter 6]{tro}
associated with $I_{\gamma ,T}^{\ast }$, namely, 
\begin{equation*}
L_{1}=\frac{d}{dt}L_{2},\quad 0=\frac{d}{dt}L_{3},
\end{equation*}%
where $L_{i}$ denotes the partial derivative of $L$ with respect to the $i$%
-th variable.

From Lemma \ref{lem:rewrite}, for $(\xi ,\zeta )\in {\mathcal{C}}_{T}$ 
\begin{equation*}
L(\xi (t),\xi ^{\prime }(t),\zeta ^{\prime }(t))=\lambda \ell
(\varphi _{1}(t))+\mu \ell (\varphi _{2}(t))+\theta \xi (t)\ell (\varphi
_{3}(t)),
\end{equation*}%
where $\varphi =(\varphi _{1},\varphi _{2},\varphi _{3})$ is given by %
\eqref{eq:varphi1} -- \eqref{eq:varphi3} when $\xi (t)\neq 0$ and by %
\eqref{eq:varphi0} when $\xi (t)=0$.

Using the form of $\varphi _{i}$, one can check that (suppressing in the
notation the dependence on $t$ and with $L_{i}=L_{i}(\xi ,\xi ^{\prime
},\zeta ^{\prime })$) 
\begin{align*}
L_{1}& =\theta \ell (\varphi _{3})+\theta \xi (\log \varphi _{3})\left( -%
\frac{\zeta ^{\prime }}{\theta \xi ^{2}}\right) =\theta (1-\varphi _{3}), \\
L_{2}& =\lambda (\log \varphi _{1})\frac{\frac{\xi ^{\prime }+\zeta ^{\prime
}}{\sqrt{(\xi ^{\prime }+\zeta ^{\prime 2}+4\lambda \mu }}+1}{2\lambda }+\mu
(\log \varphi _{2})\frac{\frac{\xi ^{\prime }+\zeta ^{\prime }}{\sqrt{(\xi
^{\prime }+\zeta ^{\prime 2}+4\lambda \mu }}-1}{2\mu }={\frac{1}{2}}\log 
\frac{\varphi _{1}}{\varphi _{2}}=\log \varphi _{1}, \\
L_{3}& =\log \varphi _{1}+\theta \xi (\log \varphi _{3})\frac{1}{\theta \xi }%
=\log (\varphi _{1}\varphi _{3}).
\end{align*}%
From $0=\frac{d}{dt}L_{3}$ we have $\log (\varphi _{1}(t)\varphi _{3}(t))
$ is a constant function of $t$, and thus for some $B\in (0,\infty )$%
\begin{equation*}
\varphi _{1}(t)=\frac{B}{\varphi _{3}(t)},\;t\in \lbrack 0,T].
\end{equation*}%
Using this equation and the equality $L_{1}=\frac{d}{dt}L_{2}$ we have $-%
\frac{\varphi _{3}^{\prime }}{\varphi _{3}}=\theta (1-\varphi _{3})$, whose
solution is given by 
\begin{equation*}
\varphi _{3}(t)=\frac{1}{1-Ae^{\theta t}},\mbox{ for some  }A\in (-\infty
,e^{-\theta T}),\;t\in \lbrack 0,T].
\end{equation*}%
Therefore 
\begin{equation}
\varphi _{1}(t)=B(1-Ae^{\theta t}),\quad \varphi _{2}(t)=\frac{1}{%
B(1-Ae^{\theta t})},\quad \varphi _{3}(t)=\frac{1}{1-Ae^{\theta t}},\quad
t\in \lbrack 0,T].  \label{eq:varphi_family}
\end{equation}%
With $A,B$ chosen to satisfy appropriate boundary conditions, the trajectory 
$(\xi ,\zeta )$ will be the candidate minimizer in the definition of $%
I_{\gamma ,T}^{\ast }$.

{Since the terminal value $\xi(T)$ in the optimization problem $I_{\gamma, T}^*$ is unspecified,} the boundary conditions are given by the transversality condition ({see \cite[Pages 145 and 156]{tro}} and the
proof of Lemma \ref{lem:minimizer-proof} for the role played by this
condition) 
\begin{equation}
L_{2}\Big|_{t=T}=0  \label{eq:transversal}
\end{equation}%
and the initial and terminal value constraints $\xi (0)=x_{0}$, $\zeta (0)=0$
and $\zeta (T)=\gamma T$. These together give the unique $(A,B)$ and the
trajectories $(\xi ,\zeta )$, as explained below. In view of Corollary \ref%
{lem:no_Skorohod}, and since $(\xi ^{\prime },\zeta ^{\prime })$ satisfy %
\eqref{eq:equation1} and \eqref{eq:equation2}, we have 
\begin{equation*}
\xi ^{\prime }+\theta \varphi _{3}\xi =\lambda \varphi _{1}-\mu \varphi _{2},
\end{equation*}%
which gives the solution 
\begin{equation*}
\frac{(1-A)e^{\theta t}}{1-Ae^{\theta t}}\xi (t)=x_{0}+\frac{e^{\theta t}-1}{%
\theta }\left[ \lambda B(1-A)-\frac{\mu }{B(1-Ae^{\theta t})}\right] ,
\end{equation*}%
namely 
\begin{equation}
\xi (t)=\frac{e^{-\theta t}-A}{1-A}x_{0}+\frac{\lambda B}{\theta }(e^{\theta
t}-1)(e^{-\theta t}-A)+\frac{\mu }{\theta B}\frac{e^{-\theta t}-1}{1-A}.
\label{eq:equation_xi}
\end{equation}%
Since $\zeta ^{\prime }$ satisfies \eqref{eq:equation2} we get 
\begin{align}
\zeta (t)& =\lambda Bt-\frac{\mu t}{B}+\frac{\mu }{\theta B}\log \frac{%
1-Ae^{\theta t}}{1-A}+\frac{1-e^{-\theta t}}{1-A}\left[ x_{0}-\frac{\lambda
B(1-A)}{\theta }+\frac{\mu }{\theta B}\right]   \notag \\
& =\frac{\lambda B}{\theta }\left[ \theta t-1+e^{-\theta t}\right] +\frac{%
\mu }{\theta B}\left[ \log \frac{e^{-\theta t}-A}{1-A}-\frac{e^{-\theta t}-A%
}{1-A}+1\right] +\frac{1-e^{-\theta t}}{1-A}x_{0}.  \label{eq:equation_zeta}
\end{align}%
From \eqref{eq:equation_xi} and \eqref{eq:equation_zeta} we have 
\begin{equation}
\xi (t)+\zeta (t)=x_{0}+\frac{\lambda B}{\theta }\left[ \theta t-A(e^{\theta
t}-1)\right] +\frac{\mu }{\theta B}\log \frac{e^{-\theta t}-A}{1-A}.
\label{eq:equation_xi+zeta}
\end{equation}%
From \eqref{eq:transversal} we have 
\begin{equation}
B(1-Ae^{\theta T})=1.  \label{eq:equation-transversal}
\end{equation}%

\begin{remark}
\label{rmk:LLN} Taking $A=0$ and $B=1$ in \eqref{eq:equation_xi} and %
\eqref{eq:equation_zeta} gives the LLN trajectory 
\begin{align}
\xi_0(t) & = e^{-\theta t} x_0 + \frac{\lambda-\mu}{\theta} (1-e^{-\theta
t}),  \label{eq:xi_LLN} \\
\zeta_0(t) & = (\lambda-\mu)t + (1-e^{-\theta t}) (x_0 - \frac{\lambda-\mu}{%
\theta}),  \label{eq:zeta_LLN}
\end{align}
with $I_T(\xi_0,\zeta_0)=0$.
\end{remark}

Equalities in \eqref{eq:equation_xi}--\eqref{eq:equation-transversal} will
determine the minimizer $(\xi, \zeta)$, and the associated $\varphi$, in the
definition of $I^*_{\gamma,T}$. Let 
\begin{equation*}
\Lambda(A,T) \doteq \frac{e^{-\theta T}-A}{1-A}.
\end{equation*}

We will now argue that for every $\gamma>0$, there is a unique $(A,B)$ that
satisfy \eqref{eq:equation_xi}--\eqref{eq:equation-transversal} together with
the terminal condition $\zeta(T)= \gamma T$. The latter condition, along
with \eqref{eq:equation_zeta} and \eqref{eq:equation-transversal}, leads to
the following equation for $A$ (see Lemma \ref{lem:termcond}): 
\begin{equation}  \label{eq:A-def}
\begin{aligned} \frac{1}{1-Ae^{\theta T}} & = \frac{1}{2\lambda \left[
\theta T - 1 + e^{-\theta T} \right]} \Big\{ \theta\left[ \gamma T - x_0 +
x_0 \Lambda(A,T) \right] \\ & \quad +\left(\theta^2\left[ \gamma T - x_0 +
x_0\Lambda(A,T) \right]^2 - 4\lambda \left[ \theta T - 1 + e^{-\theta T}
\right] \mu \left[ \log \Lambda(A,T)- \Lambda(A,T) + 1 \right]\right)^{1/2}
\Big\}. \end{aligned}
\end{equation}

The following lemma gives existence and uniqueness of solutions to the above
equation and also gives the asymptotic behavior of the corresponding $(A,B)$
as $T\rightarrow \infty $. The resulting $(\xi ,\zeta )$ introduced in
Construction \ref{constr:minimizer} below will play a key role in the
analysis.

\begin{lemma}
\label{lem:minimizer-existence} \phantomsection
Suppose $\gamma > 0$. Then the following hold.

\begin{enumerate}[(a)]

\item There exists a unique $A \in (-\infty,e^{-\theta T})$ that satisfies %
\eqref{eq:A-def}. 

\item The unique $A \equiv A(x_0, T)$, satisfies, uniformly for $x_0$ in
compacts, as $T \to \infty$, $1-Ae^{\theta T} \to \frac{2\lambda}{\gamma+%
\sqrt{\gamma^2+4\lambda\mu}} = z_\gamma$. In particular, $A \to 0$ and $B
\doteq \frac{1}{1-Ae^{\theta T}} \to z_\gamma^{-1}$, uniformly for $x_0$ in
compacts, as $T\to \infty$.
\end{enumerate}
\end{lemma}

\begin{proof}
	(a) Denote the right side of \eqref{eq:A-def} by $R(A,T)$ and the left side of the same display by $L(A,T)$.
	Since $\log x - x + 1 \le 0$ for every $x\ge 0$, $R(A,T)$ is well defined for every $A < e^{-\theta T}$.
	Let
	\begin{align*}
		f_1(A)  \doteq \frac{1}{\log \frac{1-A}{e^{-\theta T}-A}} \cdot L(A,T), \; \;
		f_2(A)  \doteq \frac{1}{\log \frac{1-A}{e^{-\theta T}-A}} \cdot R(A,T).
	\end{align*}	
	Note that (for fixed $T$) as $A \downarrow -\infty$,
	\begin{equation*}
		L(A,T) \to 0, \quad
		R(A,T) \to \frac{\theta \gamma T }{\lambda \left[ \theta T - 1 + e^{-\theta T} \right]} > 0.
	\end{equation*}
	So $f_1(A) < f_2(A)$ for sufficiently small $A$.
	Also note that as $A \uparrow e^{-\theta T}$, 
	\begin{equation}
		\label{eq:f1f2}
		L(A,T) = O\left( \frac{1}{e^{-\theta T}-A}\right), \quad
		R(A,T) = O\left( \sqrt{\log \frac{1}{e^{-\theta T}-A}}\right).
	\end{equation}
	So $f_1(A) > f_2(A)$ for $A$ sufficiently close to $e^{-\theta T}$.
	Since $f_1(A),f_2(A)$ are continuous in $A$, there must exist some $A$ such that they are equal, which gives existence of $A \in (-\infty,e^{-\theta T})$.

	For uniqueness, it suffices to verify that 
	\begin{align}
		A & \mapsto f_1(A) \mbox{ is strictly increasing in } A \in (-\infty,e^{-\theta T}), \label{eq:uniqueness_1} \\
		A & \mapsto f_2(A) \mbox{ is strictly decreasing in } A \in (-\infty,e^{-\theta T}). \label{eq:uniqueness_2}
	\end{align}
	For \eqref{eq:uniqueness_1}, since $\frac{1-A}{e^{-\theta T}-A} > 1$, it suffices to show that
	\begin{equation*}
		A \mapsto g_1(A) \doteq (e^{-\theta T}-A) \log \frac{1-A}{e^{-\theta T}-A} \mbox{ is strictly decreasing in } A \in (-\infty,e^{-\theta T}).
	\end{equation*}
	This follows on observing that
	\begin{equation*}
		g_1'(A) = \log \frac{e^{-\theta T}-A}{1-A} - \frac{e^{-\theta T}-A}{1-A} + 1 < 0.
	\end{equation*}
	This proves  \eqref{eq:uniqueness_1}.	
	For \eqref{eq:uniqueness_2}, let $z \doteq \frac{e^{-\theta T}-A}{1-A} = \Lambda(A,T)$. 
	Note that $z \in (0,1)$ and is strictly decreasing in $A \in (-\infty,e^{-\theta T})$.
	It then suffices to show
	\begin{equation*}
		z \mapsto \frac{\theta\left[ \gamma T  - x_0 + z x_0 \right] + \sqrt{\theta^2\left[ \gamma T  - x_0 + z x_0 \right]^2 - 4\lambda \left[ \theta T - 1 + e^{-\theta T} \right] \mu \left[ \log z - z + 1 \right]}}{-\log z}
	\end{equation*}
	is strictly increasing in $z \in (0,1)$.
	Since $\frac{\gamma T  - x_0 + z x_0}{-\log z}$ is strictly increasing in $z \in (0,1)$, and $x+\sqrt{x^2+C}$ is increasing in $x \in \Rmb$ for any  $C\in (0,\infty)$, it suffices to show
	\begin{equation*}
		z \mapsto \frac{- 4\lambda \left[ \theta T - 1 + e^{-\theta T} \right] \mu \left[ \log z - z + 1 \right]}{(\log z)^2} \mbox{ is increasing in } z \in (0,1).
	\end{equation*}
	Since $\theta T - 1 + e^{-\theta T} > 0$,
	it suffices in turn to show
	\begin{equation*}
		z \mapsto \frac{-(\log z - z + 1)}{(\log z)^2} \mbox{ is increasing in } z \in (0,1).
	\end{equation*}
	However, this is easy to verify, proving \eqref{eq:uniqueness_2}.
	This completes the proof of (a).

	(b) Fix a compact $K \subset \Rmb_+$.
	We first claim that
	\begin{equation*}
		\liminf_{T \to \infty} \inf_{x_0 \in K} (1-Ae^{\theta T}) \ge c_0 \mbox{ for some } c_0 > 0.
	\end{equation*}
	Indeed, note that if $\lim_{T \to \infty} Ae^{\theta T} = 1$ for some sequence $\{x_T\} \subset K$, where $A \equiv A(x_T,T)$, then
	(as $T\to \infty$)
	\begin{equation*}
		L(A,T) = O\left( \frac{1}{1-Ae^{\theta T}}\right) \mbox{ and }
		R(A,T) = O\left( \sqrt{\frac{1}{T}\log \frac{1}{e^{-\theta T}-A}}\right),
	\end{equation*}
	which says that $L(A,T)/R(A,T) \to \infty$ as $T\to \infty$ and contradicts the fact that $L(A,T)= R(A,T)$.
	Therefore the claim holds.
	Since $\gamma > 0$ and since $A = A(x_T, T)$ is the solution of \eqref{eq:A-def},
	it follows  that
	\begin{equation*}
		\liminf_{T \to \infty} \inf_{x_0 \in K} \frac{1}{1-Ae^{\theta T}} \ge \frac{\gamma}{\lambda}.
	\end{equation*}
	So for sufficiently large $T$ and uniformly for $x_0 \in K$, we have
	\begin{equation*}
		0 < \half c_0 \le 1-Ae^{\theta T} \le \frac{2\lambda}{\gamma}
	\end{equation*}
	and hence, uniformly for $x_0 \in K$,
	\begin{equation*}
		\log \Lambda(A,T) = \log \frac{e^{-\theta T}-A}{1-A} = -\theta T + \log \frac{1-Ae^{\theta T}}{1-A} = -\theta T + O(1).
	\end{equation*}
	Applying this back to \eqref{eq:A-def} gives
	\begin{equation*}
		\lim_{T \to \infty} \sup_{x_0 \in K} \left| \frac{1}{1-Ae^{\theta T}} - \frac{\gamma+\sqrt{\gamma^2+4\lambda\mu}}{2\lambda} \right| = 0.
	\end{equation*}
	This completes the proof.	
\end{proof}

\begin{construction}
\label{constr:minimizer} Let $A \in (-\infty,e^{-\theta T})$ be the unique
solution of \eqref{eq:A-def} and let $B = \frac{1}{1-Ae^{\theta T}}$. Define 
$(\bar\xi,\bar\zeta)$ as in \eqref{eq:equation_xi} and %
\eqref{eq:equation_zeta} with this choice of $(A,B)$.
\end{construction}

The following lemma shows that $(\bar \xi, \bar \zeta)$ constructed above
have the correct terminal value.

\begin{lemma}
\label{lem:termcond} Suppose $\gamma>0$. The function ${\bar{\zeta}}$
satisfies $\bar \zeta(T)=\gamma T$.
\end{lemma}

\begin{proof}
	From \eqref{eq:A-def} and the definition $B = \frac{1}{1-Ae^{\theta T}}$ we see that $B$ is one of the solution to the equation
	\begin{equation*}
		\lambda \left[ \theta T - 1 + e^{-\theta T} \right] B^2 - \theta\left[ \gamma T - x_0 + x_0 \frac{e^{-\theta T}-A}{1-A} \right] B +\mu \left[ \log \frac{e^{-\theta T}-A}{1-A} - \frac{e^{-\theta T}-A}{1-A} + 1 \right] = 0.
	\end{equation*}
	It then follows from \eqref{eq:equation_zeta} that the function $\zetabar$ satisfies $\bar \zeta(T)=\gamma T$.
\end{proof}

\subsection{Properties of the Candidate Minimizer.}

The pair $(\bar{\xi},\bar{\zeta})$ introduced in Construction \ref%
{constr:minimizer} is our candidate minimizer for $I_{\gamma ,T}^{\ast }$.
In this section we study some properties of these trajectories. In
particular we show that $(\bar{\xi},\bar{\zeta})\in {\mathcal{C}}_{T}$,
where ${\mathcal{C}}_{T}$ is defined at the beginning of Section \ref%
{sec:proprf}, when $T$ is sufficiently large. We will occasionally denote $(%
\bar{\xi},\bar{\zeta})$ as $(\bar{\xi}_{x_{0}},\bar{\zeta}_{x_{0}})$ in
order to emphasize dependence on the initial condition. Note that $(\bar{\xi}%
,\bar{\zeta})$ also depends on $T$ but that dependence is suppressed in the
notation.

The following lemma says that with $(A,B)$ as identified in Lemma \ref{lem:minimizer-existence}, the state
process ${\bar{\xi}}$ given in Construction \ref{constr:minimizer} never
goes below $0$ and actually is away from $0$ for $t > 0$.

\begin{lemma}
\label{lem:minimizer-property} Suppose $\gamma>0$. Let $({\bar{\xi}},{\bar{%
\zeta}})$ be as given in Construction \ref{constr:minimizer}. For every
compact $K \subset {\mathbb{R}}_+$, there exists some $T_0 \in (0,\infty)$
such that the following hold for every $T\ge T_0$:

\begin{enumerate}[(a)] 

\item Suppose $x_{0}\in K$. Then ${\bar{\xi}}(t)>0$ for all $t\in (0,T]$ and $({\bar{\xi}},{\bar{\zeta}})\in {%
\mathcal{C}}_{T}$.

\item Suppose $x_0 \in K \setminus \{0\}$. Then there exists some $c_0 \in
(0,\infty)$ such that $c_0^{-1} \le {\bar{\xi}}(t) \le c_0$, $|{\bar{\xi}}%
^{\prime }(t)| \le c_0$ and $c_0^{-1} \le {\bar{\zeta}}^{\prime }(t) \le c_0$
for all $t\in [0,T]$.
\end{enumerate}
\end{lemma}

\begin{proof}
	(a) From \eqref{eq:equation_xi} and since $e^{-\theta T} -A>0$, we see that, for any $t \in (0,T]$, 
	\begin{equation}
		\label{eq:minimizer-property-pf1}
		\frac{\bar\xi(t)}{1-e^{-\theta t}} \ge \frac{\lambda B}{\theta}(1-Ae^{\theta t}) - \frac{\mu}{\theta B} \frac{1}{1-A}.
	\end{equation}
	
	If $A \equiv A(x_0,T) > 0$, then the right side of \eqref{eq:minimizer-property-pf1} is decreasing in $t$.
	Therefore
	\begin{align*}
		\frac{\bar\xi(t)}{1-e^{-\theta t}} & \ge \frac{\lambda B}{\theta}(1-Ae^{\theta T}) - \frac{\mu}{\theta B} \frac{1}{1-A} \\
		&=  \frac{\lambda-\mu}{\theta}+ \frac{\mu}{\theta} - \frac{\mu}{\theta} \frac{1-Ae^{\theta T}}{1-A} = \frac{\lambda-\mu}{\theta}+\frac{\mu}{\theta} \frac{A}{1-A}(e^{\theta T}-1) > \frac{\lambda-\mu}{\theta} \ge 0,
	\end{align*}
	where the first equality on the second line uses \eqref{eq:equation-transversal} and the last inequality follows since $\lambda \ge \mu$.
	
	Now consider the case $A \le 0$.
	From Lemma \ref{lem:minimizer-existence}(b) and the fact that $z_{\gamma}^{-1} > \sqrt{\frac{\mu}{\la}}$ we can find some $T_0<\infty$ such that 
	\begin{equation}\label{eq:infinf}
		\inf_{T \ge T_0} \inf_{x_0 \in K} B > \sqrt{\frac{\mu}{\la}}.
		\end{equation}
	Therefore for all $T \ge T_0$ and $x_0 \in K$, whenever $A \le 0$, we have $B(1-A) > \sqrt{\frac{\mu}{\la}}$.
	Applying this to \eqref{eq:minimizer-property-pf1} gives
	\begin{equation*}
		\frac{\xibar(t)}{1-e^{-\theta t}} \ge \frac{\lambda B}{\theta}(1-A) - \frac{\mu}{\theta B} \frac{1}{1-A} > 0
	\end{equation*}
	for all $t \in (0,T]$. This proves the first statement in (a). The second statement is immediate from the first  since $(\xibar,\zetabar)$ satisfy \eqref{eq:equation1}--\eqref{eq:equation2}.
	This completes the proof of part (a).

	(b) From \eqref{eq:equation2} and \eqref{eq:varphi_family} we have $\zetabar'(t) = \frac{\theta \xibar(t)}{1-Ae^{\theta t}}$.
	Since $x_0>0$, from part (a) we have that $\inf_{t \in [0,T]} \xibar(t) > 0$ for all $T\ge T_0$. 
	Now for every $T \ge T_0$, from \eqref{eq:infinf} and since $-\infty < A < e^{-\theta T}$ and $\sup_{x_0\in K}|\xibar|_{*,T} < \infty$, we have that $c_1^{-1} \le \zetabar'(t) \le c_1$ for some $c_1 \in (0,\infty)$.
	Finally from \eqref{eq:equation1} and \eqref{eq:varphi_family} we see that for every $T\ge T_0$, there exists some $c_2<\infty$ such that $|\xibar'(t)| \le c_2$.
	This completes the proof.
\end{proof}

\subsection{Cost Asymptotics for the Candidate Minimizer.}

The following lemma calculates $I_T({\bar{\xi}},{\bar{\zeta}})$ for $({\bar{%
\xi}},{\bar{\zeta}})$ introduced in Construction \ref{constr:minimizer}.

\begin{lemma}
\label{lem:minimizer-cost} Suppose $\gamma > 0$. Fix a compact $K\subset {%
\mathbb{R}}_+$. Then with $T_0 \in (0,\infty)$ as in Lemma \ref%
{lem:minimizer-property}, for all $T\ge T_0$, $x_0\in K$ and $({\bar{\xi}},{%
\bar{\zeta}})$ as introduced in Construction \ref{constr:minimizer}, 
\begin{align*}
I_T({\bar{\xi}},{\bar{\zeta}}) & = (\gamma T+{\bar{\zeta}}(T)-x_0) \log B - {%
\bar{\xi}}(T) \log \frac{1}{1-Ae^{\theta T}} + x_0 \log \frac{1}{1-A} \\
& \quad + \lambda T + \mu T - \frac{\lambda B}{\theta} [\theta T -
A(e^{\theta T}-1)] + \frac{\mu}{B\theta} \log \frac{e^{-\theta T}-A}{1-A}.
\end{align*}
Furthermore 
\begin{equation*}
\lim_{T \to \infty} \frac{I_T({\bar{\xi}},{\bar{\zeta}})}{T} = C(\gamma)
\end{equation*}
uniformly for $x_0$ in $K$ where $C(\gamma)$ is defined as in \eqref{cgamma}.
\end{lemma}

\begin{proof}
	It follows from Lemma \ref{lem:minimizer-property}(a) that for  all $T\ge T_0$ and $x_0 \in K$,
	$(\xibar,\zetabar) = (\xibar_{x_0},\zetabar_{x_0})\in \Cmc_T$.
	In the following sequence of equalities we use \eqref{eq:varphi_family} for the second equality, \eqref{eq:equation1}--\eqref{eq:equation2} for the third, 
	the relation $\varphi_3'=\theta\varphi_3(\varphi_3-1)$ given above \eqref{eq:varphi_family} for the fourth, 
	and \eqref{eq:varphi_family} once more for the fifth equality. 
	\begin{align}
		\label{eq:cost}
	\begin{aligned}
		I_T(\bar\xi,\bar\zeta) & = \lambda \int_0^T \ell(\varphi_1(t))\,dt + \mu \int_0^T \ell(\varphi_2(t))\,dt + \theta \int_0^T \bar\xi(t) \ell(\varphi_3(t))\,dt \\
		& = \int_0^T \Big[ -\lambda\varphi_1(t) \log \frac{\varphi_3(t)}{B} - \lambda\varphi_1(t) + \lambda + \mu\varphi_2(t) \log \frac{\varphi_3(t)}{B} - \mu\varphi_2(t) + \mu  \\
		& \quad  + \theta\bar\xi(t)\varphi_3(t) \log \varphi_3(t) - \theta\bar\xi(t)\varphi_3(t) + \theta\bar\xi(t) \Big] dt \\
		& = \int_0^T \Big[ -\bar\xi'(t) \log \varphi_3(t) +(\bar\xi'(t)+\bar\zeta'(t)) \log B - \lambda\varphi_1(t) + \lambda\\
		&\quad  - \mu\varphi_2(t) + \mu - \theta\bar\xi(t)\varphi_3(t) + \theta\bar\xi(t) \Big] dt \\
		& = -\bar\xi(t) \log \varphi_3(t) \Big|_{t=0}^T + (\bar\xi(t)+\bar\zeta(t)) \log B \Big|_{t=0}^T + \int_0^T \left[ - \lambda\varphi_1(t) + \lambda - \mu\varphi_2(t) + \mu \right] dt \\
		& = (\bar\xi(T)+\gamma T-x_0) \log B - \bar\xi(T) \log \frac{1}{1-Ae^{\theta T}} + x_0 \log \frac{1}{1-A} \\
		& \quad + \lambda T + \mu T - \frac{\lambda B}{\theta} [\theta T - A(e^{\theta T}-1)] + \frac{\mu}{B\theta} \log \frac{e^{-\theta T}-A}{1-A}.
	\end{aligned}
	\end{align}
	This proves the first statement in the lemma.
	For the second statement, we will use Lemma \ref{lem:minimizer-existence}(b) which says that $1-Ae^{\theta T} \to z_\gamma$, $A \to 0$ and $B \to z_\gamma^{-1}$ as $T \to \infty$ uniformly for $x_0 \in K$.
	From these uniform convergence properties and  \eqref{eq:equation_xi} we have
	\begin{equation*}
		\frac{\xibar(T)}{T} = \frac{1}{T} \left( \frac{e^{-\theta T}-A}{1-A} x_0 + \frac{\lambda B}{\theta} (e^{\theta T}-1)(e^{-\theta T}-A) + \frac{\mu}{\theta B} \frac{e^{-\theta T}-1}{1-A} \right) \to 0,
	\end{equation*}
	uniformly for $x_0 \in K$, as $T \to \infty$.
	Similarly,
	\begin{align*}
		\frac{\lambda B}{\theta T} [\theta T - A(e^{\theta T}-1)] & = \lambda B - \frac{\lambda B}{\theta T} \left( Ae^{\theta T}-1+1-A \right) \to \lambda z_\gamma^{-1}, \\
		\frac{\mu}{B\theta T} \log \frac{e^{-\theta T}-A}{1-A} & = \frac{\mu}{B\theta T} \left( \log \frac{1-Ae^{\theta T}}{1-A} - \theta T \right) \to - \mu z_\gamma,
	\end{align*}
	uniformly for $x_0 \in K$, as $T \to \infty$.
	Combining these we have
	\begin{equation*}
		 \frac{I_T(\xibar,\zetabar)}{T} \to \gamma \log z_\gamma^{-1} + \lambda + \mu - \lambda z_\gamma^{-1} - \mu z_\gamma = C(\gamma),
	\end{equation*}
	uniformly for $x_0\in K$, as $T\to \infty$.
\end{proof}

\subsection{Verification of the Minimizer Property when $x_0 > 0$.}

In this section we will show that when $x_0>0$ and $\gamma>0$, $%
(\bar\xi,\bar\zeta)$ defined in Construction \ref{constr:minimizer} is the
minimizer in the variational problem for $I^*_{\gamma,T}$ when $T$ is
sufficiently large. Let 
\begin{equation}  \label{eq:Jmc}
{\mathcal{J}}(x_0, \gamma, T) \doteq \{ (\xi,\zeta) \in {\mathcal{C}}_T :
\zeta(T)=\gamma T \}.
\end{equation}
We will frequently suppress $x_0$ and $T$ from the notation and simply write 
${\mathcal{J}}(\gamma)$ for ${\mathcal{J}}(x_0, \gamma, T)$.

\begin{lemma}
\label{lem:minimizer-proof} Suppose $K\subset {\mathbb{R}}_+$ is compact and 
$\gamma > 0$. Let $(\bar\xi,\bar\zeta)$ be as in Construction \ref%
{constr:minimizer}. Then there exists some $T_1 \in(0,\infty)$ such that $%
I_{\gamma, T}^* = I_T({\bar{\xi}},{\bar{\zeta}})$ for all $T\ge T_1$ and $%
x_0 \in K \setminus \{0\}$.
\end{lemma}

\begin{proof}
	By Lemma \ref{lem:rewrite},
	\begin{align*}
		I_{\gamma, T}^* & = \inf_{(\xi,\zeta) \in \Jmc(\gamma)} I_T(\xi,\zeta) = \inf_{(\xi,\zeta) \in \Jmc(\gamma)} \int_0^T L(\xi(t),\xi'(t),\zeta'(t)) \, dt.
	\end{align*}
	Thus, in view of Lemma \ref{lem:termcond} and Lemma \ref{lem:minimizer-property}(a), it suffices to show that, with $T_0$ as in Lemma \ref{lem:minimizer-property}, there is a $T_1\ge T_0$ such that for all $T\ge T_1$ and $x_0 \in K \setminus \{0\}$
	$(\xibar,\zetabar)$ is the minimizer of the function
	\begin{equation*}
		G(\xi,\zeta) \doteq \int_0^T L(\xi(t),\xi'(t),\zeta'(t)) \, dt, \quad (\xi,\zeta) \in \Jmc(\gamma).
	\end{equation*}
	We will prove this via contradiction.
	
	First note that $\Jmc(\gamma)$ is a convex subset of $\Cmc_T$ and $G$ is a convex function on $\Jmc(\gamma)$ by Lemma \ref{lem:convexity}.
	Now suppose there exists some $(\xitil,\zetatil) \in \Jmc(\gamma)$ such that $G(\xitil,\zetatil) < G(\xibar,\zetabar)$.
	We will show that this leads to a contradiction.
	From Lemma \ref{lem:minimizer-cost},  for $T\ge T_0$
	we have $G(\xibar,\zetabar) < \infty$.
	For $\varepsilon \in [0,1]$, consider the family of paths $(\xi^\varepsilon,\zeta^\varepsilon) \doteq (1-\varepsilon) (\xibar,\zetabar) + \varepsilon (\xitil,\zetatil)$. 
	Also consider the convex function $g$ on $[0,1]$ defined by
	$g(\varepsilon) \doteq G(\xi^\varepsilon,\zeta^\varepsilon)$ for $\varepsilon \in [0,1]$.  Note that $g(1) = G(\xitil,\zetatil) < G(\xibar,\zetabar) = g(0)$.
	It follows from the convexity that $g$ is left and right differentiable wherever it is finite.
	We will show that $g'_+(0)=0$, where $g'_+(\cdot)$ denotes the right derivative of $g$.
	The convexity of $g$ will then give the desired contradiction.

	Let $\Delta \xi \doteq \xitil-\xibar$ and $\Delta \zeta \doteq \zetatil-\zetabar$.
Then
	\begin{align}
		g(\varepsilon) & = \int_0^T L(\xi^\varepsilon(t),(\xi^\varepsilon)'(t),(\zeta^\varepsilon)'(t)) \, dt \notag \\
		& = \int_0^T L(\xibar(t) + \varepsilon \Delta \xi(t), \xibar'(t) + \varepsilon \Delta \xi'(t),  \zetabar'(t) + \varepsilon \Delta \zeta'(t)) \, dt. \label{eq:minimizer-proof-g}
	\end{align}
	We claim that we can differentiate with respect to $\varepsilon$ under the integral sign for $0 < \varepsilon < \quarter$.
	Suppose for the moment that the claim is true.
	Then
	\begin{align*}
		g'(0+) & = \int_0^T \left( L_1(\xibar(t), \xibar'(t),  \zetabar'(t)) \Delta \xi(t) + L_2(\xibar(t), \xibar'(t),  \zetabar'(t)) \Delta \xi'(t) \right. \\
		& \qquad \quad \left. +  L_3(\xibar(t), \xibar'(t),  \zetabar'(t)) \Delta \zeta'(t) \right) dt
	\end{align*}
	Recall from Section \ref{sec:formal} that $(\xibar,\zetabar)$ is such that $L_1=\frac{d}{dt}L_2$, $0=\frac{d}{dt}L_3$ and $L_2\big|_{t=T} = 0$ (when evaluated at $(\bar \xi, \bar \xi', \bar \zeta, \bar \zeta')$).
	It then follows from integration by parts that
	\begin{align*}
		g'(0+) & = \left( L_2(\xibar(t), \xibar'(t),  \zetabar'(t)) \Delta \xi(t) + L_3(\xibar(t), \xibar'(t),  \zetabar'(t)) \Delta \zeta(t) \right) \Big|_{t=0}^{t=T}.
	\end{align*} 
	Since $\Delta \xi(0) = 0$, $L_2(\xibar(T), \xibar'(T),  \zetabar'(T)) = 0$, $\Delta \zeta(0)=0=\Delta \zeta(T)$, we have $g'(0+)=0$.
	This gives the desired contradiction and shows that $(\xibar,\zetabar)$ is the minimizer.

	Finally we prove the claim that in \eqref{eq:minimizer-proof-g} we can differentiate under the integral sign for $0 < \varepsilon < \quarter$.
	Denote the integrand in \eqref{eq:minimizer-proof-g} by $\tilde g(t, \varepsilon)$, namely
	\begin{equation*}
		\gtil(t,\varepsilon) \doteq L(\xibar(t) + \varepsilon \Delta \xi(t), \xibar'(t) + \varepsilon \Delta \xi'(t),  \zetabar'(t) + \varepsilon \Delta \zeta'(t)).
	\end{equation*}
	It suffices to establish an integrable bound on $\frac{\partial \gtil(t,\varepsilon)}{\partial \varepsilon}$ that is uniform for $0 < \varepsilon < \quarter$.	
	From the formula for $L$ in Section \ref{sec:proprf} and recalling that $\xi^{\varepsilon}>0$, we have
	\begin{align*}
		\frac{\partial \gtil(t,\varepsilon)}{\partial \varepsilon} & = L_1(\xi^\varepsilon(t),(\xi^\varepsilon)'(t),(\zeta^\varepsilon)'(t)) \Delta \xi(t) + L_2(\xi^\varepsilon(t),(\xi^\varepsilon)'(t),(\zeta^\varepsilon)'(t)) \Delta \xi'(t) \\
		& \quad +  L_3(\xi^\varepsilon(t),(\xi^\varepsilon)'(t),(\zeta^\varepsilon)'(t)) \Delta \zeta'(t) \\
		& = \Delta \xi(t) \left( \theta - \frac{(\zeta^\varepsilon)'(t)}{\xi^\varepsilon(t)} \right) + \Delta \zeta'(t) \log \frac{(\zeta^\varepsilon)'(t)}{\theta \xi^\varepsilon(t)} \\
		& \quad + \left( \Delta \xi'(t) + \Delta \zeta'(t) \right) \log \frac{\sqrt{((\xi^\varepsilon)'(t) + (\zeta^\varepsilon)'(t))^2 + 4\lambda\mu} + (\xi^\varepsilon)'(t) + (\zeta^\varepsilon)'(t)}{2\lambda}.
	\end{align*}
	From Lemma \ref{lem:minimizer-property}(b) we have (by choosing a larger $T_0$ if needed) $\delta \doteq \inf_{t \in [0,T]} \xibar(t) > 0$ (we remark that $\delta$ may depend on $x_0$).
	Then there exists some $\kappa \in (0,\infty)$ such that for all  $0 < \varepsilon < \quarter$ and $t \in [0,T]$,
	\begin{equation}
		\label{eq:xi_vareps_bd}
		\frac{\delta}{2} < \xi^\varepsilon(t) \le \kappa.
	\end{equation}	
	It then follows from continuity of $\xitil$ and $\xibar$ that
	\begin{equation}
		\label{eq:interchange1}
		\left| \Delta \xi(t) \left( \theta - \frac{(\zeta^\varepsilon)'(t)}{\xi^\varepsilon(t)} \right) \right| \le \kappa_1(1 + \zetatil'(t) + \zetabar'(t)).
	\end{equation} 
	Next, by Lemma \ref{lem:minimizer-property}(b) and \eqref{eq:xi_vareps_bd}, 
	\begin{align}
		\label{eq:secinteg}
		\left| \Delta \zeta'(t) \log \frac{(\zeta^\varepsilon)'(t)}{\theta \xi^\varepsilon(t)} \right| \le \kappa_2  +  \zetatil'(t) \log \zetatil'(t).
	\end{align}
	Since $G(\xitil,\zetatil) < G(\xibar,\zetabar) < \infty$, we must have $\zetabar', \zetatil'$ and $\zetatil'(\cdot) \log \zetatil'(\cdot)$
	are all in $ L^1([0,T])$.
	Therefore, the right side of the above display and the right side of  \eqref{eq:interchange1} are integrable functions $[0,T]$. 
	Finally, by Lemma \ref{lem:minimizer-property}(b) 
	\begin{align*}
		& \left| \left( \Delta \xi'(t) + \Delta \zeta'(t) \right) \log \frac{\sqrt{((\xi^\varepsilon)'(t) + (\zeta^\varepsilon)'(t))^2 + 4\lambda\mu} + (\xi^\varepsilon)'(t) + (\zeta^\varepsilon)'(t)}{2\lambda} \right| \\
		& \le \kappa_3 + \kappa_3 |\xitil'(t)+\zetatil'(t))| \log \frac{\sqrt{(\xitil'(t)+\zetatil'(t))^2 + 4\lambda\mu} + \xitil'(t)+\zetatil'(t)}{2\lambda}.
	\end{align*}	
	Again, since $G(\xitil,\zetatil) < G(\xibar,\zetabar) < \infty$, we must have that the right side of the above display is an integrable function on $[0,T]$.
	Combining this with the integrability of right sides of \eqref{eq:interchange1} and \eqref{eq:secinteg} we have an
	integrable uniform bound on $\frac{\partial \gtil(t,\varepsilon)}{\partial \varepsilon}$ for $\varepsilon \in (0, \quarter)$.
	This completes the proof.		
\end{proof}

\section{Proof of Theorem \protect\ref{thm:I*}.}

In this section we will complete the proof of Theorem \ref{thm:I*}.
Throughout the section $\lambda \ge \mu$. 
{Although Lemmas  \ref{lem:minimizer-cost} and \ref{lem:minimizer-proof} together immediately imply the first statement of Theorem \ref{thm:I*}, when $x_0>0$ and $\gamma>0$, they cannot be directly used to treat the remaining cases.
In order to treat the general case we will analyze $\limsup_{T \to \infty} \frac{I_{\gamma, T}^*}{T}$ for $\gamma \ge 0$ and $x_0\ge 0$ in Proposition \ref{prop:I*limsup} and $\liminf_{T \to \infty} \frac{I_{\gamma, T}^*}{T}$ in Proposition \ref{prop:I*liminf_gamma_positive} (for $\gamma > 0$ and $x_0\ge 0$ ) and Proposition \ref{prop:I*liminf_gamma_zero} (for $\gamma =0$ and $x_0\ge 0$).
These three propositions taken together cover the general case $x_0 \ge 0$ and $\gamma \ge 0$.
The proof of Proposition \ref{prop:I*liminf_gamma_positive} will make crucial use of Lemmas  \ref{lem:minimizer-cost} and \ref{lem:minimizer-proof}.}

We begin by establishing some useful properties of $I_{\gamma, T}^*$.

\begin{lemma}
\label{lem:IT_convex} The map $\gamma \mapsto I_{\gamma, T}^*$ is a
nonnegative, convex function on $[0,\infty)$, and $I_{\gamma^*_T,T}^*=0$
where $\gamma^*_T = (\lambda-\mu) + \frac{1-e^{-\theta T}}{T}(x_0 - \frac{%
\lambda-\mu}{\theta})$. In particular, the function $\gamma \mapsto
I_{\gamma, T}^*$ is decreasing for $\gamma \le \gamma^*_T$ and increasing
otherwise.
\end{lemma}

\begin{proof}
	Clearly $I_{\gamma, T}^* \ge 0$ for every $\gamma\ge 0$.
	From Remark \ref{rmk:LLN} we see that the LLN zero-cost trajectory $(\xi_0,\zeta_0)$ has total reneging $(\lambda-\mu)T + (1-e^{-\theta T})(x_0 - \frac{\lambda-\mu}{\theta})$ over $[0,T]$ which proves that $I_{\gamma^*_T,T}^*=0$.
	In order to see convexity, fix $\gamma_2 > \gamma_1 \ge 0$ and $\varepsilon \in (0,1)$.
	Let $\gamma \doteq \varepsilon\gamma_1 + (1-\varepsilon)\gamma_2$.
	For  $(\xi_i,\zeta_i) \in \Jmc(\gamma_i)$, $i=1,2$, let 
	\begin{equation*}
		(\xi,\zeta) \doteq \varepsilon (\xi_1,\zeta_1) + (1-\varepsilon) (\xi_2,\zeta_2).
	\end{equation*}
	Then $(\xi,\zeta) \in \Jmc(\gamma)$ and by Lemma \ref{lem:convexity}
	\begin{align*}
		I_{\gamma, T}^* \le I_T(\xi,\zeta) \le \varepsilon I_T(\xi_1,\zeta_1) + (1-\varepsilon) I_T(\xi_2,\zeta_2).
	\end{align*}
	Taking infimum over $(\xi_i,\zeta_i) \in \Jmc(\gamma_i)$, $i=1,2$, gives
	\begin{equation*}
		I_{\gamma, T}^* \le \varepsilon I_{\gamma_1,T}^* + (1-\varepsilon) I_{\gamma_2,T}^*.
	\end{equation*}
	The final monotonicity statement in the lemma is now immediate from convexity and nonnegativity.
\end{proof}
%
%
%
%
%
%
%
%
Note that for $\gamma \ge 0$, $z_{\gamma}$ given in \eqref{eq:z} is the
unique positive solution to 
\begin{equation}  \label{eq:z-equation}
\lambda z_\gamma^{-1} - \mu z_\gamma - \gamma = 0.
\end{equation}
Also note that the map $\gamma \mapsto C(\gamma)$ is continuous and $%
C(\lambda-\mu)=0$. The following proposition bounds the limit superior of $%
I^*_{\gamma,T}/T$ by $C(\gamma)$.

\begin{proposition}
\label{prop:I*limsup} Fix $\gamma \ge 0$. Then 
\begin{equation*}
\limsup_{T \to \infty} \frac{I_{\gamma, T}^*}{T} \le C(\gamma).
\end{equation*}
\end{proposition}

\begin{proof}
	First consider the case that $\gamma>0$.
	For fixed $T$, we make use of the trajectory $(\xi,\zeta)$ given by \eqref{eq:equation_xi} and \eqref{eq:equation_zeta}, with
	$A=0$ and $B$ chosen such that $\zeta(T)= \gamma T$.
	This means that
	\begin{align*}
		B  &= \frac{\theta\left[\gamma T + (e^{-\theta T}-1) x_0\right] + \sqrt{\theta^2\left[\gamma T + (e^{-\theta T}-1) x_0\right]^2 + 4\lambda\mu \left[ \theta T - 1 + e^{-\theta T} \right]^2}}{2\lambda \left[ \theta T - 1 + e^{-\theta T} \right]}, \\
		\xi(t)  &= e^{-\theta t} x_0 + \frac{\lambda B}{\theta} (1-e^{-\theta t}) - \frac{\mu}{\theta B} (1-e^{-\theta t}), \; 0 \le t \le T, \\
		\zeta(t) &= \frac{\lambda B}{\theta} \left[ \theta t - 1 + e^{-\theta t} \right] + \frac{\mu}{\theta B} \left[ -\theta t - e^{-\theta t} + 1 \right] + (1-e^{-\theta t}) x_0.
	\end{align*}
	Since $\gamma>0$, we have
	\begin{equation*}
		\lim_{T \to \infty} B  = \frac{\gamma+\sqrt{\gamma^2 + 4\lambda\mu}}{2\lambda} = z_\gamma^{-1} > \sqrt{\frac{\mu}{\lambda}}.
	\end{equation*}
	This implies $B > \sqrt{\frac{\mu}{\lambda}}$, $\xi(t) \ge 0$, and $(\xi,\zeta) \in \Cmc_T$ for sufficiently large $T$, and
	\begin{align*}
		\lim_{T \to \infty}	\xi(T) = \frac{\lambda}{z_{\gamma}\theta} - \frac{z_\gamma\mu}{\theta}.
	\end{align*}
	We note that since  \eqref{eq:cost} is satisfied for any $\varphi$ satisfying \eqref{eq:varphi_family} and the corresponding $(\bar \xi, \bar \zeta)$ satisfying \eqref{eq:equation1}--\eqref{eq:equation2}, this equality holds with
	 $(\bar \xi, \bar \zeta)$ replaced by $(\xi, \zeta)$. 
	 Using this and since $\zeta(T)=\gamma T$, $A=0$,
	\begin{align*}
		\frac{I_{\gamma, T}^*}{T} & \le \frac{\xi(T)+\gamma T-x_0}{T} \log B - \frac{\xi(T)}{T} \log \frac{1}{1-Ae^{\theta T}} + \frac{x_0}{T} \log \frac{1}{1-A} \\
		& \quad + \lambda + \mu - \frac{\lambda B}{\theta T} [\theta T - A(e^{\theta T}-1)] + \frac{\mu}{B\theta T} \log \frac{e^{-\theta T}-A}{1-A}\\
		&= \frac{\xi(T)+\gamma T-x_0}{T} \log B  
		 + \lambda + \mu - \lambda B - \frac{\mu}{ B}
	\end{align*}
	for sufficiently large $T$,
	and hence
	\begin{equation*}
		\limsup_{T \to \infty} \frac{I_{\gamma, T}^*}{T} \le \gamma \log z_\gamma^{-1}  + \lambda + \mu - \lambda z_\gamma^{-1} - \mu z_\gamma = C(\gamma).
	\end{equation*}
	
	Next consider the case that $\gamma=0$. Let $z_{-1} = \frac{2\lambda}{\sqrt{4\lambda \mu +1}-1}$.
	Take 
	\begin{align*}
		& \xi(t)=x_0-t, \: \zeta(t)=0, \: \varphi_1(t)=z_{-1}^{-1}, \: \varphi_2(t)=z_{-1}, \: \varphi_3(t)=0, \quad t \in [0,x_0), \\
		& \xi(t)=0, \: \zeta(t)=0, \: \varphi_1(t)=z_0^{-1}, \: \varphi_2(t)=z_0, \: \varphi_3(t)=1, \quad t \in [x_0,T].
	\end{align*}
	Clearly $(\xi,\zeta) \in \Cmc_T$.
	Using \eqref{eq:z-equation} we see that $\lambda \varphi_1(t) - \mu \varphi_2(t) = -1 = \xi'(t)$ for $t \in [0,x_0)$ and $\lambda \varphi_1(t) - \mu \varphi_2(t) = 0 = \xi'(t)$ for $t \in [x_0,T]$.
	Therefore $(\varphi_1,\varphi_2,\varphi_3) \in \Umc(\xi,\zeta)$ and
	\begin{align*}
		\limsup_{T \to \infty} \frac{I_{0, T}^*}{T} & \le \limsup_{T \to \infty} \frac{\lambda \int_0^T \ell(\varphi_1(t))\,dt + \mu\int_0^T \ell(\varphi_2(t))\,dt + \theta\int_0^T \xi(t)\ell(\varphi_3(t))\,dt}{T} \\
		& = \lambda \ell(z_0^{-1}) + \mu \ell(z_0) = \lambda(1-z_0^{-1}) + \mu(1-z_0) = C(0).
	\end{align*}	
	This completes the proof.
\end{proof}

The following proposition gives the reverse inequality for limit inferior
when $\gamma > 0$.

\begin{proposition}
\label{prop:I*liminf_gamma_positive} Fix $\gamma > 0$. Then 
\begin{equation*}
\liminf_{T \to \infty} \frac{I_{\gamma, T}^*}{T} \ge C(\gamma).
\end{equation*}
\end{proposition}

\begin{proof}
	First consider $x_0 > 0$.
	Then it is immediate from Lemmas  \ref{lem:minimizer-cost} and \ref{lem:minimizer-proof} that 
	$$\lim_{T \to \infty} \frac{I_{\gamma, T}^*}{T} = C(\gamma).$$
	
	Next consider $x_0=0$. There are two cases for $\gamma$:
	
	Case 1: $\gamma = \lambda-\mu$. In this case $C(\gamma)=0$ and the result clearly holds.
	
	Case 2: $\gamma \ne \lambda-\mu$. 
	Let $$\Jmc(x,\gamma,T) \doteq\{ (\xi,\zeta) \in \Cmc_T : \xi(0)=x, \zeta(T) = \gamma T\}.$$
	Fix $\delta \in (0,|\gamma-(\lambda-\mu)|\wedge 1)$ and
	 $T \in (0,\infty)$. 
	There exists $(\xi^T,\zeta^T) \in \Jmc(0,\gamma,T)$ such that $I_{\gamma, T}^* + \delta \ge I_T(\xi^T,\zeta^T)$.	
	Let $(\xibar,\zetabar)$ be as in Construction \ref{constr:minimizer} with $x_0=0$.
	By Lemma \ref{lem:minimizer-property}(a) and Lemma \ref{lem:termcond} there exists $T_1 \in (0,\infty)$ such that $(\xibar,\zetabar) \in \Jmc(0,\gamma,T)$ for $T \ge T_1$.
	Also note that, by Lemma \ref{lem:minimizer-cost},
	\begin{equation*}
		\lim_{T \to \infty} \frac{I_T(\xibar,\zetabar)}{T} = C(\gamma).
	\end{equation*}
	Consider $\Jmc(0,\gamma,T) \ni (\xitil^T,\zetatil^T) \doteq \delta (\xibar,\zetabar) + (1-\delta) (\xi^T,\zeta^T)$.
	By Lemma \ref{lem:convexity}, we have
	\begin{align*}
		I_T(\xitil^T,\zetatil^T) \le \delta I_T(\xibar,\zetabar) + (1-\delta) I_T(\xi^T,\zeta^T) \le \delta I_T(\xibar,\zetabar) + I_T(\xi^T,\zeta^T)
	\end{align*}
	and hence
	\begin{equation}
		\label{eq:I*liminf_1}
		\liminf_{T \to \infty} \frac{I_{\gamma, T}^*}{T} \ge \liminf_{T \to \infty} \frac{I_T(\xi^T,\zeta^T)}{T} \ge \liminf_{T \to \infty} \frac{I_T(\xitil^T,\zetatil^T)}{T} - \delta C(\gamma).
	\end{equation}
	Now let $$\sigma_T \doteq \inf\{t: \xitil^T(t) \ge 1 \mbox{ or } \zetatil^T(t) \ge 1 \} \wedge 1.$$
	By restricting the attention to the cost over $[\sigma_T,T]$ we have
	\begin{equation*}
		I_T(\xitil^T,\zetatil^T) \ge \inf_{(\xi,\zeta) \in \Jmc(\xitil^T(\sigma_T),\frac{\gamma T - \zetatil^T(\sigma_T)}{T-\sigma_T},T-\sigma_T)} I_{T-\sigma_T}(\xi,\zeta).
	\end{equation*}
	Let $i=1$ if $\gamma > \lambda-\mu$ and $i=-1$ if $\gamma < \lambda-\mu$.
	Noting that $\delta \in (0,|\gamma-(\lambda-\mu)|\wedge 1)$ and $\xitil^T(\sigma_T), \zetatil^T(\sigma_T), \sigma_T \in [0,1]$, using the last statement in Lemma \ref{lem:IT_convex} we can find $T_2 \in(T_1, \infty)$ such that for all $T \ge T_2$,
	\begin{equation*}
		\inf_{(\xi,\zeta) \in \Jmc(\xitil^T(\sigma_T),\frac{\gamma T - \zetatil^T(\sigma_T)}{T-\sigma_T},T-\sigma_T)} I_{T-\sigma_T}(\xi,\zeta) \ge \inf_{(\xi,\zeta) \in \Jmc(\xitil^T(\sigma_T),\gamma-i\delta,T-\sigma_T)} I_{T-\sigma_T}(\xi,\zeta).
	\end{equation*}
	Let $(\xibar^T,\zetabar^T) \in \Jmc(\xitil^T(\sigma_T),\gamma-i\delta,T-\sigma_T)$ be given by Construction \ref{constr:minimizer}.
	Noting that $\sigma_T>0$, from Lemma \ref{lem:minimizer-property}(a) we have $\xibar(\sigma_T) > 0$ and hence $0 < \xitil^T(\sigma_T) \le 1$.
	It then follows from Lemma \ref{lem:minimizer-proof} that
	\begin{equation*}
		\inf_{(\xi,\zeta) \in \Jmc(\xitil^T(\sigma_T),\gamma-i\delta,T-\sigma_T)} I_{T-\sigma_T}(\xi,\zeta) = I_{T-\sigma_T}(\xibar^T,\zetabar^T)
	\end{equation*}
	for all $T \ge T_3$ for some $T_3 \in (T_2,\infty)$.
	From Lemma \ref{lem:minimizer-cost} (applied with $K=[0,1]$)
	\begin{equation*}
		\lim_{T \to \infty} \frac{I_{T-\sigma_T}(\xibar^T,\zetabar^T)}{T-\sigma_T} = C(\gamma-i\delta).
	\end{equation*}
	Combining the last four displays we have
	\begin{equation*}
		\liminf_{T \to \infty} \frac{I_T(\xitil^T,\zetatil^T)}{T} \ge \liminf_{T \to \infty} \frac{I_{T-\sigma_T}(\xibar^T,\zetabar^T)}{T} = C(\gamma-i\delta).
	\end{equation*}
	Combining this with \eqref{eq:I*liminf_1}, taking $\delta \to 0$ and using the continuity of $C(\gamma)$ completes the proof.
\end{proof}

The following proposition gives the analogue of Proposition \ref%
{prop:I*liminf_gamma_positive} when $\gamma = 0$.

\begin{proposition}
\label{prop:I*liminf_gamma_zero} 
\begin{equation*}
\liminf_{T \to \infty} \frac{I_{0,T}^*}{T} \ge C(0).
\end{equation*}
\end{proposition}

\begin{proof}
	It follows from Lemma \ref{lem:IT_convex} and Proposition \ref{prop:I*liminf_gamma_positive} that
	\begin{equation*}
		\liminf_{T \to \infty} \frac{I_{0,T}^*}{T} \ge \liminf_{T \to \infty} \frac{I_{\varepsilon,T}^*}{T} \ge C(\varepsilon)
	\end{equation*}
	for each $\varepsilon \in (0,\lambda-\mu)$.
	The result follows on sending $\varepsilon \to 0$.
\end{proof}

We can now complete the proof of Theorem \ref{thm:I*}.

\textbf{Proof of Theorem \ref{thm:I*}.} The first statement
follows from Propositions \ref{prop:I*limsup}, \ref%
{prop:I*liminf_gamma_positive}, \ref{prop:I*liminf_gamma_zero}. For the
second and third statements, note that by Lemma \ref{lem:IT_convex} the
following holds for sufficiently large $T$: 
\begin{align*}
I_{\gamma,T} & = \inf \{ I_T(\xi,\zeta) : \zeta(T) \ge \gamma T \} = \inf_{{%
\tilde{\gamma}} \ge \gamma} I_{{\tilde{\gamma}},T^*} = I_{\gamma, T}^*,
\quad \gamma > \lambda - \mu, \\
{\tilde{I}}_{\gamma,T} & \doteq \inf \{ I_T(\xi,\zeta) : \zeta(T) \le \gamma
T \} = \inf_{{\tilde{\gamma}} \le \gamma} I_{{\tilde{\gamma}},T}^* =
I_{\gamma, T}^*, \quad 0 \le \gamma < \lambda-\mu.
\end{align*}
This proves these two statements when $\gamma \neq \lambda-\mu$. The result
for the case when $\gamma = \lambda-\mu$ is immediate on observing that 
\begin{equation*}
0 \le I_{\gamma,T} \le I_{\gamma, T}^*, \quad 0 \le {\tilde{I}}_{\gamma,T}
\le I_{\gamma, T}^*, \quad \lim_{T \to \infty} \frac{I_{\gamma, T}^*}{T} = 0
= C(\gamma).
\end{equation*}
\hfill \qed
\begin{remark}
	{We now provide a heuristic interpretation of the minimizers $(\bar\xi,\bar\zeta)$ given in Construction
	\ref{constr:minimizer}. First observe that in the case there is no reneging (i.e. $\theta=0$), the law of large numbers
	limit of the state process $X^n$ is given as the solution of the ordinary differential equation (ODE):
	$$\dot \xi_0(t) = (\lambda - \mu), \; \xi_0(0)=x_0.$$
	Thus the limit is simply given as the trajectory with initial value the scaled limit of of initial queue length and velocity given as the difference between the inflow rate and the outflow rate. In the case where the reneging occurs at rate $\theta$, the
	LLN limit of $X^n$ is given in \eqref{eq:xi_LLN} which can be rewritten as
	$$\dot \xi_0(t) = - \theta \xi_0(t) + (\lambda - \mu), \; \xi_0(0)=x_0.$$
	Thus the only difference is that the velocity is decreased by a factor of $\theta \xi_0(t)$ to account for rate $\theta$ reneging.
	Now fix $\gamma >0$. Then the minimizing trajectories over the time interval $[0,T]$  associated with  an asymptotic reneging 
	rate of $\gamma$ are given as in  Construction
	\ref{constr:minimizer}. Recall from  Lemma \ref{lem:minimizer-existence}(b) that the parameters $A,B$ (which depend on $T$) associated with these trajectories satisfy $A\to 0$, $B \to z_\gamma^{-1}$ and 
	$1-Ae^{\theta T} \to   z_\gamma$. The trajectory $\xi^*$ obtained by replacing in the definition of $\bar \xi$  these asymptotic values of $A$ and $B$ is given as the solution of the ODE
	$$\dot \xi^*(t) = - \theta \xi^*(t) + (\lambda z_\gamma^{-1} - \mu z_\gamma), \; \xi^*(0)=x_0.$$
Thus, formally speaking, in order for the queue to experience a long-term reneging rate of $\gamma$, the arrival rates and service rates behave atypically according to $\lambda z_\gamma^{-1}$ [resp.\ $\mu z_\gamma$] instead of
their nominal values $\lambda $ [resp.\ $\mu$]. Recall that when $\gamma= \lambda-\mu$, $z_{\gamma}=1$ in which case $\xi^*$ corresponds to the LLN limit.  }
\end{remark}
\section{Multi-Server Queue.}

\label{sec:multiser} The techniques developed in this paper for the analysis
of the $M/M/1+M$ model extend to the multiserver setting, namely to the $%
M/M/n+M$ model. In this section we outline this extension without proofs.
{
As in the case with the single server model, the
arrival rate considered is scaled up by $n$,
whereas the individual reneging rate remains fixed. However,
the service station now consists of $n$ exponential servers
each serving at a fixed rate
(this is often referred to in the literature as the many-server
scaling; see \cite{atakasshi},  \cite{kanram10}).
Again let $U^n$ and $V^n$ denote the number in system
process and the cumulative reneging count, respectively.
Then the queue length process is now given by $(U^n-n)^+$, and
the number of customers in service by $U^n\w n$.
As before, the rescaled versions are denoted by
$X^n=U^n/n$ and $Y^n=V^n/n$.
}

As with the $M/M/1+M$ setting it is convenient to introduce certain PRM.
Replace $N_2$ in Section \ref{sec:setting} by a PRM on $[0,T]\times {\mathbb{%
R}}_+^2$ with intensity $\mu ds \times dy \times dz$ and denote it once more
by $N_2$. Let ${\mathbb{X}}_2 \doteq {\mathbb{R}}_+^2$. Let $\bar{\mathcal{A}%
}_i$ for $i=1,3$ be as before and $\bar{\mathcal{A}}_2= \bar{\mathcal{A}}_3$%
. For $\varphi \in \bar {\mathcal{A}}_i$, $i=1,3$, $N^{\varphi}_i$ are
defined as before and let $N^{\varphi}_2$ be defined using $N_2$ and $%
\varphi $ as $N^{\varphi}_3$ is defined using $N_3$ and $\varphi$.
The processes $(X^n, Y^n)$
are given as follows. 
\begin{align*}
X^n(t) &= x_n + \frac{1}{n} \int_{[0,t]} N_1^n(ds) - \frac{1}{n}
\int_{[0,t]\times {\mathbb{R}}_+} {\boldsymbol{1}}_{[0, X^n(s-)\wedge
1]}(y)\, N_2^n(ds\, dy) \\
& \quad - \frac{1}{n} \int_{[0,t]\times {\mathbb{R}}_+} {\boldsymbol{1}}%
_{[0, \left(X^n(s-) - 1\right)^+]}(y) \,N_3^n(ds\,dy). \\
Y^n(t) &= \frac{1}{n} \int_{[0,t]\times {\mathbb{R}}_+} {\boldsymbol{1}}%
_{[0, \left(X^n(s-) - 1\right)^+]}(y) \,N_3^n(ds\,dy).
\end{align*}
{
On the RHS of the first equation, the three integrals correspond
to the arrival, departure and reneging processes, respectively,
normalized by $n$. Thinning with indicator of $[0,X^n(s-)\w1]$
corresponds to the fact that service rate is proportional to the number
of customers in service; indeed, normalizing $U^n\w n$ by $n$
gives $X^n\w1$.
Similarly, the expression in the third integral accounts for the fact
that the total reneging rate is proportional
to the queue length, $(U^n-n)^+$, that after normalization is given by
$(X^n-1)^+$.
}

The LLN of $(X^n, Y^n)$ is governed by the following equations 
\begin{equation}  \label{eq:fluidlimitmsq}
\begin{aligned} x(t) &= x_0 + \lambda t -\mu \int_0^{t} (x(s)\wedge 1) ds -
\theta \int_0^{t} (x(s)-1)^+ \,ds, \; t \in [0,T], \\ y(t) &= \theta
\int_0^t (x(s)-1)^+ \,ds, \; t \in [0,T]. \end{aligned}
\end{equation}

We now introduce the rate function associated with LDP for $(X^n, Y^n)$. For 
$(\xi, \zeta) \in \mathcal{C}([0,T]: {\mathbb{R}}_+^2)$, let $\mathcal{U}%
(\xi, \zeta)$ be the collection of all $\varphi = (\varphi_1, \varphi_2,
\varphi_3)$ such that $\varphi_i: [0,T] \to {\mathbb{R}}_+$, $i=1,2,3$ are
measurable maps and the following equations are satisfied for all $t \in
[0,T]$: 
\begin{equation}  \label{eq:eqxizetmsq}
\begin{aligned} \xi(t) &= x_0 + \lambda\int_0^{t} \varphi_1(s)\,ds - \mu
\int_0^{t} (\xi(s)\wedge 1) \varphi_2(s)\,ds - \theta \int_0^{t}
\varphi_3(s) (\xi(s) -1)^+\,ds, \\ \zeta(t) &= \theta \int_0^t
\varphi_3(s)(\xi(s) -1)^+\,ds. \end{aligned}
\end{equation}
For $(\xi, \zeta) \in \mathcal{C}([0,T]: {\mathbb{R}}_+^2)$, define 
\begin{equation}  \label{eq:I_Tmsq}
I_T(\xi, \zeta) \doteq \inf_{\varphi \in \mathcal{U}(\xi, \zeta)}\left\{
\lambda\int_0^T \ell(\varphi_1(s))\,ds + \mu\int_0^T (\xi(s)\wedge 1)
\ell(\varphi_2(s))\,ds + \theta\int_0^T (\xi(s)
-1)^+\ell(\varphi_3(s))\,ds\right\}.
\end{equation}
Set $I_T(\xi, \zeta)$ to be $\infty$ if $\mathcal{U}(\xi, \zeta)$ is empty
or $(\xi, \zeta) \in \mathcal{D}([0,T]: {\mathbb{R}}_+^2)\setminus \mathcal{C%
}([0,T]: {\mathbb{R}}_+^2)$. The following result can be established using
similar methods as for the proof of Theorem \ref{th:thm1}.

\begin{theorem}
\label{th:thm1msq} The pair $\{(X^n, Y^n)\}$ satisfies a LDP on $\mathcal{D}%
([0,T]: {\mathbb{R}}_+^2)$ with rate function $I_T$.
\end{theorem}

Once more using contraction principle one obtains a LDP for $\zeta(T)$,
namely Theorem \ref{thm:contra} holds with $I^*_T$ as in \eqref{eq:istart}
and $I_T$ as in \eqref{eq:I_Tmsq}. For the analysis of calculus of
variations problem we assume for simplicity that $\lambda \ge \mu$ and that $%
x_0 \ge 1$. As 
in Lemma \ref{lem:rewrite} we can give an alternative representation for the
rate function $I_T$ in terms of a local rate function $L$. In calculating
the minimizer for $I^*_{\gamma,T}$ one considers again Euler-Lagrange
equations as in Section \ref{sec:formal}. One sees that in finding the
minimizer one can restrict attention to trajectories $\xi$ that satisfy $%
\xi(t)\ge 1$ for all $t$. In fact the minimizer for $\gamma >0$ is given by $%
(\bar \xi^*, \bar \zeta^*) = (\bar \xi_{(x_0-1)} + 1 , \bar \zeta_{(x_0-1)})$
where $(\bar \xi_{(x_0-1)} , \bar \zeta_{(x_0-1)})$ is given by Construction %
\ref{constr:minimizer} with $x_0$ replaced by $x_0-1$. Using this, one can
show that Theorem \ref{thm:I*} holds for the $M/M/n+M$ model with the same
decay rate $C(\gamma)$. We omit the details.

%
%
%

\section*{Acknowledgments.}

We thank two referees for carefully reading our paper and their helpful criticisms and suggestions.
Research of RA is supported in part by the Israel Science Foundation (grant 1184/16).
Research of AB is supported in part by the National Science Foundation (DMS-1814894 and DMS-1853968).
Research of PD is supported in part by the Air Force Office of Scientific Research AFOSR (FA-9550-18-1-0214) and the National Science Foundation (DMS-1904992).



\end{document}